\documentclass[
final
]{dmtcs-episciences}

\usepackage[utf8]{inputenc}
\usepackage{subfigure}
\usepackage{tikz}
\usepackage{amsthm,amsmath,amssymb,bold-extra}

\usetikzlibrary{decorations.pathreplacing}
\newtheorem{thm}{Theorem}[section]
\newtheorem{lma}[thm]{Lemma}
\newtheorem{cor}[thm]{Corollary}

\newtheorem{prop}[thm]{Proposition}

\newtheorem*{adef}{Definition}
\newtheorem*{claim*}{Claim}

\DeclareMathOperator{\grd}{grd}

\numberwithin{equation}{section}

\usepackage[square,numbers]{natbib}

\author{Kitty Meeks\affiliationmark{1}\thanks{Supported by a Royal Society of Edinburgh Personal Research Fellowship, funded by the Scottish Government.}
  \and Dominik K. Vu\affiliationmark{2}}
\title{Extremal properties of flood-filling games}
\affiliation{
  School of Computing Science, University of Glasgow\\
  University of Memphis}
\received{2018-3-29}
\revised{2019-3-8}
\accepted{2019-7-3}

\begin{document}
\publicationdetails{21}{2019}{4}{12}{4412}
\maketitle
\begin{abstract}
The problem of determining the number of ``flooding operations'' required to make a given coloured graph monochromatic in the one-player combinatorial game Flood-It has been studied extensively from an algorithmic point of view, but basic questions about the maximum number of moves that might be required in the worst case remain unanswered.  We begin a systematic investigation of such questions, with the goal of determining, for a given graph, the maximum number of moves that may be required, taken over all possible colourings.  We give several upper and lower bounds on this quantity for arbitrary graphs and show that all of the bounds are tight for trees; we also investigate how much the upper bounds can be improved if we restrict our attention to graphs with higher edge-density.
\end{abstract}

%%%%%%%%%% SECTION 1: INTRODUCTION %%%%%%%%%%
\section{Introduction}

Flood-It is a one-player combinatorial game, played on a coloured graph.  The goal is to make the entire graph monochromatic (``flood'' the graph) with as few moves as possible, where a move involves picking a vertex $v$ and a colour $d$, and giving all vertices in the same monochromatic component as $v$ colour $d$, as illustrated in Figure \ref{flooding-example}.  Implementations of the game played on regular grids are widely available online \cite{newflash,madvirus} and as popular smartphone apps.  More generally, when played on a planar graph, the game can be regarded as modelling repeated use of the flood-fill tool in Microsoft Paint.

\begin{figure}[h]
\centering
\begin{tikzpicture}[scale=.38]

\draw (0,0) circle (.75); \draw node at (0,0) {1}; \draw (2,0) circle (.75); \draw node at (2,0) {2}; \draw (4,0) circle (.75); \draw node at (4,0) {4}; \draw (0,-2) circle (.75); \draw node at (0,-2) {2}; \draw (2,-2) circle (.75); \draw node at (2,-2) {3}; \draw (0,-4) circle (.75); \draw node at (0,-4) {4}; \draw (2,-4) circle (.75); \draw node at (2,-4) {1}; \draw (4,-4) circle (.75); \draw node at (4,-4) {3}; 
\draw[-] (0.75,0) -- (1.25,0);\draw[-] (2.75,0) -- (3.25,0);
\draw[-] (0.75,-4) -- (1.25,-4);\draw[-] (2.75,-4) -- (3.25,-4);
\draw[-] (0,-0.75) -- (0,-1.25);\draw[-] (0,-2.75) -- (0,-3.25);
\draw[-] (2,-0.75) -- (2,-1.25);\draw[-] (2,-2.75) -- (2,-3.25);
\draw[-] (4,-0.75) -- (4,-3.25);
\draw[-] (0.53,-2.53) -- (1.47,-3.47);\draw[-] (2.53,-1.47) -- (3.47,-0.53);

\draw[->] (5,-2) -- (6.5,-2);

\draw (8,0) circle (.75); \draw node at (8,0) {2}; \draw (10,0) circle (.75); \draw node at (10,0) {2}; \draw (12,0) circle (.75); \draw node at (12,0) {4}; \draw (8,-2) circle (.75); \draw node at (8,-2) {2}; \draw (10,-2) circle (.75); \draw node at (10,-2) {3}; \draw (8,-4) circle (.75); \draw node at (8,-4) {4}; \draw (10,-4) circle (.75); \draw node at (10,-4) {1}; \draw (12,-4) circle (.75); \draw node at (12,-4) {3}; 
\draw[-] (8.75,0) -- (9.25,0);\draw[-] (10.75,0) -- (11.25,0);
\draw[-] (8.75,-4) -- (9.25,-4);\draw[-] (10.75,-4) -- (11.25,-4);
\draw[-] (8,-0.75) -- (8,-1.25);\draw[-] (8,-2.75) -- (8,-3.25);
\draw[-] (10,-0.75) -- (10,-1.25);\draw[-] (10,-2.75) -- (10,-3.25);
\draw[-] (12,-0.75) -- (12,-3.25);
\draw[-] (8.53,-2.53) -- (9.47,-3.47);\draw[-] (10.53,-1.47) -- (11.47,-0.53);

\draw[->] (13,-2) -- (14.5,-2);

\draw (16,0) circle (.75); \draw node at (16,0) {1}; \draw (18,0) circle (.75); \draw node at (18,0) {1}; \draw (20,0) circle (.75); \draw node at (20,0) {4}; \draw (16,-2) circle (.75); \draw node at (16,-2) {1}; \draw (18,-2) circle (.75); \draw node at (18,-2) {3}; \draw (16,-4) circle (.75); \draw node at (16,-4) {4}; \draw (18,-4) circle (.75); \draw node at (18,-4) {1}; \draw (20,-4) circle (.75); \draw node at (20,-4) {3}; 
\draw[-] (16.75,0) -- (17.25,0);\draw[-] (18.75,0) -- (19.25,0);
\draw[-] (16.75,-4) -- (17.25,-4);\draw[-] (18.75,-4) -- (19.25,-4);
\draw[-] (16,-0.75) -- (16,-1.25);\draw[-] (16,-2.75) -- (16,-3.25);
\draw[-] (18,-0.75) -- (18,-1.25);\draw[-] (18,-2.75) -- (18,-3.25);
\draw[-] (20,-0.75) -- (20,-3.25);
\draw[-] (16.53,-2.53) -- (17.47,-3.47);\draw[-] (18.53,-1.47) -- (19.47,-0.53);

\draw[->] (21,-2) -- (22.5,-2);

\draw (24,0) circle (.75); \draw node at (24,0) {4}; \draw (26,0) circle (.75); \draw node at (26,0) {4}; \draw (28,0) circle (.75); \draw node at (28,0) {4}; \draw (24,-2) circle (.75); \draw node at (24,-2) {4}; \draw (26,-2) circle (.75); \draw node at (26,-2) {3}; \draw (24,-4) circle (.75); \draw node at (24,-4) {4}; \draw (26,-4) circle (.75); \draw node at (26,-4) {4}; \draw (28,-4) circle (.75); \draw node at (28,-4) {3}; 
\draw[-] (24.75,0) -- (25.25,0);\draw[-] (26.75,0) -- (27.25,0);
\draw[-] (24.75,-4) -- (25.25,-4);\draw[-] (26.75,-4) -- (27.25,-4);
\draw[-] (24,-0.75) -- (24,-1.25);\draw[-] (24,-2.75) -- (24,-3.25);
\draw[-] (26,-0.75) -- (26,-1.25);\draw[-] (26,-2.75) -- (26,-3.25);
\draw[-] (28,-0.75) -- (28,-3.25);
\draw[-] (24.53,-2.53) -- (25.47,-3.47);\draw[-] (26.53,-1.47) -- (27.47,-0.53);

\draw[->] (29,-2) -- (30.5,-2);

\draw (32,0) circle (.75); \draw node at (32,0) {3}; \draw (34,0) circle (.75); \draw node at (34,0) {3}; \draw (36,0) circle (.75); \draw node at (36,0) {3}; \draw (32,-2) circle (.75); \draw node at (32,-2) {3}; \draw (34,-2) circle (.75); \draw node at (34,-2) {3}; \draw (32,-4) circle (.75); \draw node at (32,-4) {3}; \draw (34,-4) circle (.75); \draw node at (34,-4) {3}; \draw (36,-4) circle (.75); \draw node at (36,-4) {3}; 
\draw[-] (32.75,0) -- (33.25,0);\draw[-] (34.75,0) -- (35.25,0);
\draw[-] (32.75,-4) -- (33.25,-4);\draw[-] (34.75,-4) -- (35.25,-4);
\draw[-] (32,-0.75) -- (32,-1.25);\draw[-] (32,-2.75) -- (32,-3.25);
\draw[-] (34,-0.75) -- (34,-1.25);\draw[-] (34,-2.75) -- (34,-3.25);
\draw[-] (36,-0.75) -- (36,-3.25);
\draw[-] (32.53,-2.53) -- (33.47,-3.47);\draw[-] (34.53,-1.47) -- (35.47,-0.53);

\end{tikzpicture}
\caption{A sequence of moves to flood a graph; the colour of the monochromatic component containing the top left vertex is changed at each move.}\label{flooding-example}
\end{figure}
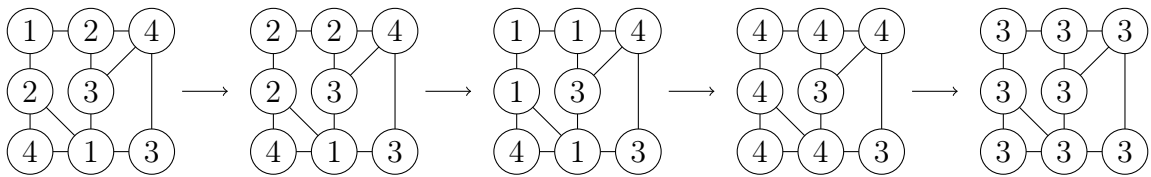

Questions arising from this game (and a two-player variant known as the Honey-Bee Game) have received considerable attention from a complexity-theoretic perspective in recent years \cite{belmonte18, clifford, daSilva18, fellows-flood15, fellows18, fleischer10, fukui, fukui13, kaihon15, lagoutte11, general,2xn,spanning,souza14-graphpowers}, with such work focussing on questions of the form, ``Given a graph $G$ from a specified class, and a colouring $\omega$ of the vertices of $G$, what is the computational complexity of determining the minimum number of moves required to flood $G$?''  The problem is known to be NP-hard in many situations, provided that at least three colours are present in the initial colouring, including in the case that $G$ is an $n \times n$ grid (as in the original version of the game) \cite{clifford} and the case in which $G$ is a tree \cite{fleischer10,lagoutte13} (the parameterised complexity of the problem restricted to trees has also been studied \cite{fellows-flood15}).  On the other hand, there are polynomial-time algorithms to determine the minimum number of moves required if $G$ is a path or a cycle, or more generally for any graph if the initial colouring uses only two colours.  A more complete description of the complexity landscape for these problems can be found in a recent survey \cite{fellows-survey}.

Two different versions of the game have been considered in the literature, known as the ``fixed'' and ``free'' versions.  In the fixed version of the game (as in most implementations), players must always change the colour of the monochromatic component containing a single distinguished \emph{pivot vertex}, so the only choice is what colour to assign to this component; in the free version players can choose freely at each move the component whose colour is changed, in addition to the new colour.  

In this paper, we initiate a systematic investigation of a different type of question about the game, raised by Meeks and Scott \cite{2xn} but as yet unanswered in the literature: given a graph $G$, what is the maximum number of moves we may need to flood $G$, taken over all possible colourings of $G$ with $c$ colours?  In addition to providing a deeper insight into the behaviour of the game, questions of this form are motivated by recent algorithmic work involving Integer Programming formulations for the optimisation problem, whose running-times might be reduced by better bounds on the worst-case number of moves required.

In Section \ref{general-bounds} we obtain a number of straightforward upper and lower bounds on the maximum number of moves that might be required in the worst case in both the fixed and free versions of the game; perhaps surprisingly, we are also able to demonstrate that all of these simple bounds are tight for suitable families of trees.  It follows from previous work \cite{spanning} relating the number of moves required to flood a graph to the number of moves required to flood its spanning trees that any bound on the worst-case number of moves required is tight if and only if it is tight for some family of trees, but we might be able to obtain much better upper bounds on the number of moves required if we know our graph is far from being a tree: intuitively, we expect that a single colouring cannot simultaneously be the worst possible for all spanning trees if the graph has many spanning trees.  In Section \ref{dense-graphs} we investigate this issue, and in the process we determine the worst-case number of moves required to flood a graph that is a blow-up of a long path.  

In the remainder of this section, we introduce some key notation and definitions, and mention some results from the existing literature that are of particular relevance to addressing extremal problems.

\subsection{Notation and definitions}

For any graph $G=(V,E)$, we denote by $|G|$ the number of vertices in $G$ (so $|G| = |V|$).  Throughout this paper, we consider only connected graphs: if $G$ is disconnected then it is impossible to flood $G$ in the fixed version of the game, and in the free version an optimal strategy for the entire graph is obtained by playing optimal strategy in each connected component.  We write $\mathcal{T}(G)$ for the set of spanning trees of $G$.  

For $u,v \in V(G)$, we let $\mathcal{P}(u,v)$ be the set of $u$-$v$ paths in $G$, and the \emph{distance} $d(u,v)$ from $u$ to $v$ is defined to be $\min_{P \in \mathcal{P}(u,v)} |P| - 1$.  The \emph{eccentricity} of a vertex $v \in V(G)$ is $\max_{v \in V(G)} d(u,v)$, and the \emph{radius} of $G$ is the minimum eccentricity of any vertex in $G$, that is $\min_{u \in V(G)} \max_{v \in V(G)} d(u,v)$.  

Let $A \subset V(G)$.  We write $G[A]$ for the subgraph of $G$ induced by $A$, and $N(A)$ for the set of vertices in $V(G) \setminus A$ with at least one neighbour in $A$.

A graph $G = (V_G,E_G)$ is said to be a \emph{blow-up} of a graph $H = (V_H,E_H)$ if $V_G$ can be partitioned into sets $\{V_u : u \in V_H\}$ such that $v_1v_2 \in E_G$ if and only if $v_1 \in V_u$ and $v_2 \in V_w$ with $uw \in E_H$.  

Suppose the game is played on a graph $G$, equipped with an initial colouring $\omega: V \rightarrow C$ (not necessarily a proper colouring); we call $C$ the \emph{colour-set}.  We say that $u$ and $v$ belong to the same \emph{monochromatic component} of $G$ with respect to $\omega$ if there is a path $u = x_1,x_2,\ldots,x_{\ell}=v$ in $G$ such that $\omega(x_1) = \omega(x_2) = \cdots = \omega(x_{\ell})$.  In the fixed version, a move consists of a single colour $d$, and involves assigning colour $d$ to all vertices in the same monochromatic component as the pivot vertex $u$; in the free version, a move $m=(v,d)$ involves assigning colour $d$ to all vertices in the same monochromatic component as $v$.  Given a colouring $\omega$ with colour-set $C$ and any $d \in C$, we denote by $N_d(G,\omega)$ the number of vertices $v \in V$ such that $\omega(v) = d$.

For any graph $G$ with colouring $\omega$, we can obtain a new graph and corresponding colouring by \emph{contracting monochromatic components} of $G$ with respect to $\omega$, that is repeatedly contracting an edge $e = uv$ such that $\omega(u) = \omega(v)$.  If $G'$ (with colouring $\omega '$) is obtained from $G$ (with initial colouring $\omega$) in this way, it is clear that any sequence of moves that floods $G$ with initial colouring $\omega$ will also flood $G'$ with initial colouring $\omega'$, and vice versa (up to possibly changing the vertex at which a move is played to another vertex in the same monochromatic component with respect to $\omega$).  We say that the graph $G_1$ with colouring $\omega_1$ is \emph{equivalent} to the graph $G_2$ with colouring $\omega_2$ if there is a colour-preserving isomorphism (i.e. a bijection between vertices that preserves (non-)adjacency and colours) from the coloured graph obtained by contracting monochromatic components of $G_1$ with respect to $\omega_1$ to that obtained by contracting monochromatic components of $G_2$ with respect to $\omega_2$.

We define $m_G(G,\omega,d)$ to be the minimum number of moves required to give all vertices of $G$ colour $d$ in the free version of the game, and $m_G(G,\omega)$ to be $\min_{d \in C}m_{G}(G,\omega,d)$.  Analogously, for the fixed version of the game we write $m_G^{(v)}(G,\omega,d)$ to be the minimum number of moves required to give all vertices of $G$ colour $d$ when all moves are played at the pivot vertex $v$, and define $m_G^{(v)}(G,\omega)$ to be $\min_{d \in C}m_{G}^{(v)}(G,\omega,d)$.  

Let $\Omega(V,C)$ be the set of all surjective functions from $V$ to $C$.  We then define
$$M_c(G) = \max_{\omega \in \Omega(V,\{1,\ldots,c\})} m_G(G,\omega),$$ 
so $M_c(G)$ is the maximum number of moves that might be required in the free version to flood $G$ in the worst case, taken over all possible initial colourings with $c$ colours. We define $M_c^{(v)}(G)$ analogously for the fixed version.

Let $A$ be any subset of $V$.  We set $m_G(A,\omega,d)$ to be the minimum number of moves we must play in $G$ (with initial colouring $\omega$) in the free version to create a monochromatic component of colour $d$ that contains every vertex in $A$, and $m_G(A,\omega)=\min_{d \in C}m_G(A,\omega,d)$.  When the ground graph and area to be flooded agree, we may henceforth omit the subscript.  

We say a move $m = (v,d)$ is \emph{played in} $A$ if $v \in A$, and that $A$ is \emph{linked} if it is contained in a single monochromatic component.  Subsets $A,B \subseteq V$ are \emph{adjacent} if there exists $ab \in E$ with $a \in A$ and $b \in B$.  We will use the same notation when referring to (the vertex-set of) a subgraph $H$ of $G$ as for a subset $A \subseteq V(G)$.

\subsection{Background results}

One key result which we will exploit throughout this paper gives a characterisation of the number of moves required to flood a graph in terms of the number of moves required to flood its spanning trees. More precisely we will apply the following result by Meeks and Scott~\cite{spanning}.

\begin{thm}
Let $G$ be a connected graph with colouring $\omega$ from colour-set $C$.  Then, for any $d \in C$,
$$m(G,\omega,d) = \min_{T \in \mathcal{T}(G)} m(T,\omega,d).$$
\label{spanning-trees}
\end{thm}

A corollary of this result, proved in the same paper, is that the number of moves required to flood a graph $H$ cannot be increased when moves are played in a larger graph $G$ which contains $H$ as a subgraph.

\begin{cor}
Let $G$ be a connected graph with colouring $\omega$ from colour-set $C$, and $H$ a connected subgraph of $G$.  Then, for any $d \in C$,
$$m_G(V(H),\omega,d) \leq m_H(H,\omega,d).$$
\label{subgraph}
\end{cor}
We will also use a simple monotonicity result for paths, proved by the same authors in a previous paper~\cite{general}.

\begin{lma}
Let $P$ be a path, with colouring $\omega$ from colour-set $C$, and let $P'$ be a second coloured path with colouring $\omega'$, obtained from $P$ by deleting one vertex and joining its neighbours.  Then, for any $d \in C$, $m(P',\omega',d) \leq m(P,\omega,d)$.  We also have $m(P',\omega') \leq m(P,\omega)$.
\label{basic-monotonicity}
\end{lma}

Another useful result, proved in an additional paper by Meeks and Scott~\cite{2xn}, relates the number of moves required to flood the same graph with different initial colourings.

\begin{lma}
Let $G$ be a connected graph, and let $\omega$ and $\omega'$ be two colourings of the vertices of $G$ (from colour-set $C$).  Let $\mathcal{A}$ be the set of all maximal monochromatic components of $G$ with respect to $\omega'$, and for each $A \in \mathcal{A}$ let $c_A$ be the colour of $A$ under $\omega'$.  Then, for any $d \in C$,
$$m(G,\omega,d) \leq m(G,\omega',d) + \sum_{A \in \mathcal{A}} m(A,\omega,c_A).$$
\label{change-colouring}
\end{lma}

This result means that we do not normally need to worry about the possible effect that moves played to flood a particular subgraph might have elsewhere.

%%%%%%%%%% SECTION 2: UPPER BOUNDS %%%%%%%%%%
\section{General Bounds}
\label{general-bounds}

In this section we obtain some general lower and upper bounds on the value of $M_c(G)$ and $M_c^{(v)}(G)$, where $G$ is an arbitrary connected graph.  While these bounds are based on simple observations about the flooding process, we are able to show that all of them are tight for suitably chosen trees.

\subsection{Lower bounds}

We begin with a simple observation about the minimum number of moves required to flood a graph coloured with $c$ colours.

\begin{prop}
Let $G$ be any graph with colouring $\omega$, where $\omega$ uses exactly $c$ colours on $G$.  Then
$$m(G,\omega) \geq c-1.$$
Moreover, if every colour appears in at least two distinct monochromatic components with respect to $\omega$, then
$$m(G,\omega) \geq c.$$
\label{c-col}
\end{prop}
\begin{proof}
To see that the first statement is true, note that any move can reduce the number of colours present in the graph by at most one, and so, in order to reduce the total number of colours present from $c$ to $1$, at least $c-1$ moves are required.  For the second part of the result, observe that the first move played can only change the colour of one monochromatic component under the initial colouring, and so if every colour initially appears in at least two distinct monochromatic components then the first move cannot reduce the total number of colours present on the graph; thus a total of at least $c$ moves will be required to reduce the number of colours in the graph to $1$.
\end{proof}

This gives an immediate lower bound on $M_c(G)$.  A vertex $v$ is said to be a \emph{dominating vertex} if every other vertex is adjacent to $v$.

\begin{cor}
For any connected graph $G$, $M_c(G) \geq c-1$.  If $G$ has no dominating vertex, then $M_c(G) \geq c$.
\end{cor}

The same reasoning allows us to make a slightly stronger statement in the fixed case.

\begin{prop}
For any connected graph $G = (V,E)$ and $v \in V$, $M^{(v)}_c(G) \geq c-1$.  If $v$ is not a dominating vertex, then $M^{(v)}_c(G) \geq c$.
\label{c-col-fixed}
\end{prop}

To see that this first pair of bounds is tight, consider the complete bipartite graph $K_{1,n}$ with any colouring using $c$ colours: we can always flood this graph by playing changing the colour of $v$, the unique non-leaf vertex, $c-1$ times.  Now suppose that $u$ is a leaf in $K_{1,n}$; we obtain a new graph $G$ by adding an additional vertex $x$ adjacent only to $u$.  Now, for any colouring $\omega$ of $G$ with $c$ colours such that $\omega(x) \neq \omega(u)$ and $\omega(x) = \omega(v)$, we can flood $G$ by changing the colour of $v$ precisely $c$ times.

In order for the lower bounds we have obtained thus far to be tight, we required that some monochromatic component is adjacent to many others.  We formalise this observation with a second general lower bound based on the structure maximum number of vertices adjacent to any connected subgraph.  Note that this maximum number of neighbours will be bounded by a constant when the graph in question is obtained from some fixed graph by subdividing edges.

\begin{prop}
Let $G=(V,E)$ be a connected graph on $n \geq 1$ vertices, and let $\omega$ be a proper colouring of $G$.  Suppose that, for every set $A \subset V$ such that $G[A]$ is connected, we have $|N(A)| \leq s$.  Then 
$$m(G,\omega) \geq \frac{1}{s}(n - 1).$$
\label{edges-bound}
\end{prop}
\begin{proof}
We proceed by induction on $m(G,\omega)$; the result is trivially true for $m(G,\omega) = 0$, so assume that $m(G,\omega) > 0$.  Let $S$ be a sequence which floods $G$ with some colour $d$, where $|S| = m(G,\omega)$; suppose that the first move of $S$ is $\alpha = (v,d')$, and let $\omega'$ denote the colouring of $G$ obtained by playing $\alpha$ from the initial colouring $\omega$.  Let $A$ denote the monochromatic component of $G$ with respect to $\omega'$ that contains $v$.  It is clear that every vertex of $A$ other than $v$ must belong to some monochromatic component with respect to $\omega$ that is adjacent to $A$; since all monochromatic components with respect to $\omega$ are singletons (as $\omega$ is a proper colouring) it follows that $A \subset v \cup N(\{v\})$.  Note that $\{v\}$ trivially induces a connected subgraph of $G$, so by assumption we have $|N(\{v\})| \leq s$, implying that $|A| \leq 1 + s$.

Now consider the graph $G'$ obtained from $G$ by contracting $A$ to a single vertex; let $\omega''$ be the colouring of $G'$ derived from $\omega'$, and note that $\omega''$ is a proper colouring of $G'$.  It is also clear that contracting edges cannot create a set $B$ such that $G[B]$ is connected and $|N(B)| > s$, so we may apply the inductive hypothesis to see that
$$m(G',\omega'') \geq \frac{1}{s}(n - |A|) \geq \frac{1}{s}(n - s - 1) = \frac{1}{s}(n-1) - 1.$$
Thus we can conclude that $|S| \geq 1 + \frac{1}{s}(n-1) - 1 = \frac{1}{s}(n-1)$, as required.
\end{proof}

Again, the same argument can be applied to make a slightly stronger statement in the fixed case.

\begin{prop}
Let $G=(V,E)$ be a connected graph on $n \geq 1$ vertices, fix $v \in V$, and let $\omega$ be a proper colouring of $G$.  Suppose that, for every set $A \subset V$ such that $v \in A$ and $G[A]$ is connected, we have $|N(A)| \leq s$.  Then 
$$m^{(v)}(G,\omega) \geq \frac{1}{s}(n - 1).$$
\label{edges-bound-fixed}
\end{prop}

We can immediately deduce lower bounds on $M_c(G)$ and $M^{(v)}_c(G)$.

\begin{cor}
Let $G = (V,E)$ be a connected graph and suppose that, for every set $A \subset V$ such that $G[A]$ is connected, we have $|N(A)| \leq s$.  Then 
$$M_c(G) \geq \frac{1}{s}(n-1).$$
If, for some $v \in V$, we have $|N(A)| \leq s$ for all $A \subset V$ such that $v \in A$ and $G[A]$ is connected, then 
$$M^{(v)}_c(G) \geq \frac{1}{s}(n-1).$$
\end{cor}

To see that these bounds are tight, consider a path on $2t + 1$ vertices for some $t \in \mathbb{N}$, whose vertices are coloured alternately with colours $1$ and $2$.  Here we have $s = 2$, and we can flood the path by changing colour of the midpoint exactly $t$ times.

\subsection{Upper bounds}

We start with a simple bound based on the initial colouring of $G$.

\begin{prop}
Let $G$ be any connected graph with colouring $\omega$ from colour-set $C$.  Then, for any $d \in C$,
$$m(G,\omega,d) \leq n - N_d(G,\omega).$$
\label{max-colour-class}
\end{prop}
\begin{proof}
This result follows immediately from the fact that, provided at least one vertex does not yet have colour $d$, it is always possible to play a move which increases the number of vertices having colour $d$ by at least one: changing the colour of a vertex that does not already have colour $d$ to $d$ may additionally give some other vertices colour $d$, but all vertices that previously had colour $d$ will be unchanged.
\end{proof}

This result can also be reformulated in terms of the number of colours used in the initial colouring.

\begin{prop}
Let $G$ be any connected graph.  Then
$$M_c(G) \leq n - \left \lceil \frac{n}{c} \right \rceil.$$
\label{colour-bound}
\end{prop}
\begin{proof}
Let $\omega$ be any colouring of $G$ with $c$ colours.  There must then be at least one colour $d$ such that $N_d(G,\omega) \geq \left \lceil \frac{n}{c} \right \rceil$, implying by Proposition \ref{max-colour-class} that $m(G,\omega,d) \leq n - \left \lceil \frac{n}{c} \right \rceil$.  
\end{proof}

In the fixed version of the game, we do not necessarily benefit from having many vertices coloured with the same colour initially.  Consider a path on $n$ vertices, whose vertices are coloured alternately with colours $1$ and $2$, and fix the pivot to be one of the endpoints.  By Proposition \ref{edges-bound-fixed}, we can see that we will require at least $n-1$ moves to flood the graph, even though half the vertices already have the same colour.

Returning our attention to the free version, we can show that the two simple upper bounds above are in fact tight for suitably coloured paths; we start by defining a useful family of colourings for paths.

\begin{adef} 
Let $P = v_1 \ldots v_n$ be a path with edge-set $E = \{v_iv_{i+1}: 1 \leq i \leq n-1\}$, $C = \{d_0,\ldots,d_{c-1}\}$ a set of colours, and $\omega : V(P) \rightarrow C$ a proper colouring of $P$.  The colouring $\omega$ is said to be a  $C$-\emph{rainbow} colouring of $P$ if there exists a permutation $\pi: \{0,\ldots,c-1\} \rightarrow \{0,\ldots,c-1\}$ such that, for $1 \leq i \leq n$, $\omega(v_i) = d_{\pi(i \mod c)}$.
\end{adef}

Note that, up to relabelling of the colours, a $C$-rainbow coloured path must be as illustrated in Figure \ref{rainbow-path}.  We say that a colouring $\omega$ of $P$ is a \emph{rainbow colouring} if it is a $C$-rainbow colouring for some colour-set $C$.

\begin{figure}[h]
\centering
\begin{tikzpicture}[scale=.5]

\draw (0,0) circle (.75); \draw node at (0,0) {$d_0$}; \draw (2,0) circle (.75); \draw node at (2,0) {$d_1$}; \draw (4,0) circle (.75); \draw node at (4,0) {$d_2$}; \draw (8,0) circle (.75); \draw node at (8,0) {$d_{c-1}$}; \draw (10,0) circle (.75); \draw node at (10,0) {$d_0$}; \draw (18,0) circle (.75); \draw node at (18,0) {$d_{c-1}$}; \draw (20,0) circle (.75); \draw node at (20,0) {$d_0$};  \draw[-] (0.75,0) -- (1.25,0);\draw[-] (2.75,0) -- (3.25,0);\draw[-] (8.75,0) -- (9.25,0);\draw[-] (18.75,0) -- (19.25,0); \draw[dotted] (4.75,0)--(7.25,0); \draw[dotted] (10.75,0) -- (17.25,0); \draw[dotted] (20.75,0) -- (22.25,0); \draw (23,0) circle(.75); \draw node at (23,0) {$d_i$};
\end{tikzpicture}
\caption{A $C$-rainbow colouring of a path}\label{rainbow-path}
\end{figure}
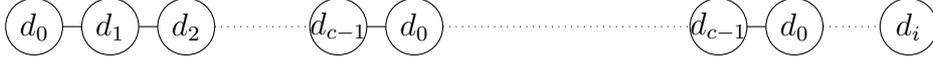

We now demonstrate that any rainbow colouring of a path attains the upper bound from Proposition \ref{max-colour-class}. 

\begin{lma}
Let $P$ be a path on $n$ vertices, and $\omega$ a rainbow colouring of $P$.  Then, for any colour $d$,
$$m(P,\omega,d) \geq n - N_d(P,\omega).$$
\label{rainbow-target}
\end{lma}
\begin{proof}
We proceed by induction on $m(P,\omega,d)$.  For the base case, suppose that $m(P,\omega,d) = 0$, which is only possible if the path is already monochromatic with colour $d$; thus $N_d(P,\omega) = n$ and so $n - N_d(P,\omega) = 0 \leq m(P,\omega,d)$, as required.

Now suppose that $m(P,\omega,d) > 0$ and that the result holds for any path $P'$, colouring $\omega'$ and colour $d'$ such that $m(P',\omega',d') < m(P,\omega,d)$.  Let $S$ be an optimal sequence of moves to flood $P$ with colour $d$ (starting from the initial colouring $\omega$), and let $\alpha$ be the final move of $S$.  There are now two cases to consider, depending on whether or not $P$ is monochromatic immediately before $\alpha$ is played.

Suppose first that $P$ is monochromatic in some colour $d' \neq d$ immediately before $\alpha$ is played.  In this case we know that $m(P,\omega,d') \leq m(P,\omega,d) - 1$ and so we may apply the inductive hypothesis to see that 
$$m(P,\omega,d') \geq n - N_{d'}(P,\omega).$$
Note that the number of vertices with colours $d$ and $d'$ in a rainbow colouring can differ by at most one, so we see that
\begin{align*}
m(P,\omega,d) & \geq m(P,\omega,d') + 1 \\
                & \geq n - N_{d'}(P,\omega) + 1 \\
                & \geq n - (N_d(P,\omega) + 1) + 1 \\
                & = n - N_d(P,\omega),
\end{align*}
as required.

Now suppose that $P$ is not monochromatic immediately before $\alpha$ is played.  In this case, before the final move, there must be either two or three monochromatic segments; we denote these segments $P_1,\ldots,P_\ell$ (where $\ell \in \{2,3\}$), and without loss of generality we may assume that $P_2$ does not have colour $d$ before $\alpha$ is played.  For each $i \in \{1,\ldots,\ell\}$, let $S_i$ be the subsequence of $S \setminus \alpha$ consisting of moves played in a monochromatic component that intersects $P_i$; note that these subsequences partition $S \setminus \alpha$.  Moreover, observe that $S_i$, played in $P_i$, must make $P_i$ monochromatic; for each $i \neq 2$ the sequence $S_i$ must flood $P_i$ with colour $d$, while $S_2$ must flood $P_2$ with some colour $d' \neq d$.  Thus we see that, for $i \neq 2$, 
$$m(P_i,\omega,d) \leq |S_i|,$$
and also
$$m(P_2,\omega,d') \leq |S_2|.$$
Observe further that $\omega$ is a rainbow colouring of $P_i$ so, as $|S_i| \leq |S \setminus \alpha| < m(P,\omega,d)$ for all $i$, we can then apply the inductive hypothesis to see that, for $i \neq 2$,
$$|S_i| \geq m(P_i,\omega,d) \geq |P_i| - N_d(P_i,\omega),$$
and that
$$|S_2| \geq m(P_2,\omega,d') \geq |P_2| - N_{d'}(P_2,\omega) \geq |P_2| - N_d(P_2,\omega) - 1$$
as the number of vertices that initially have colours $d$ and $d'$ can differ by at most one.  

This then implies that
\begin{align*}
m(P,\omega,d) & = |S| \\
	          & = 1 + \sum_{i = 1}^\ell |S_i| \\
	          & = 1 + |S_2| + \sum_{i \neq 2} |S_i| \\
	          & \geq 1 + |P_2| - N_d(P_2,\omega) - 1 + \sum_{i \neq 2} (|P_i| - N_d(P_i,\omega)) \\
	          & = |P| - N_d(P,\omega),
\end{align*}
as required.
\end{proof}

As the number of occurrences of each colour in a rainbow colouring of an $n$-vertex path with $c$ colours is either $\lfloor \frac{n}{c} \rfloor$ or $\left \lceil \frac{n}{c} \right\rceil$, we see that Proposition \ref{colour-bound} is also tight.

\begin{cor}
Let $P_t$ denote the path on $t$ vertices, and assume that $t \geq c$.  Then 
$$M_c(P_t) = t - \left \lceil \frac{t}{c} \right \rceil.$$
Moreover, this maximum possible number of moves is obtained with a rainbow colouring.
\label{path-result}
\end{cor}

Since the tightness of these bounds is demonstrated by means of a long path, it is natural to ask how much better we can do if we exclude the presence of such paths; in the next result we show that we can bound the number of moves required in the fixed version in terms of only the total number of colours and the length of the longest induced path starting at the pivot vertex.

\begin{prop}
Let $G = (V,E)$ be a connected graph, fix $v \in V$, and suppose that $v$ has eccentricity $r$,  Then
$$M_c^{(v)}(G) \leq (c-1)r.$$
\label{eccentricity-bound}
\end{prop}
\begin{proof}
For $1 \leq i \leq r$, let $V_i = \{u \in V: d(u,v) = i\}$, and note that $V = \{v\} \cup \bigcup_{1 \leq i \leq r} V_i$.  Now fix any colouring $\omega \in \Omega(V,\{1,\ldots,c\})$.  We will argue, by induction on $r$, that there is a sequence of moves played at $v$ which will flood the graph.  The base case, for $r=0$, is trivial, so we will assume that $r > 0$ and that the result holds for all graphs with radius smaller than $r$.

Let $C_1$ be the set of colours, other than $\omega(v)$, that occur at vertices of $V_1$ under $\omega$; note that $|C_1| \leq c-1$.  Then, cycling $v$ through all colours in $C_1$ will create a monochromatic component containing (at least) all of $\{v\} \cup V_1$; we will denote the new colouring of $G$ resulting from these moves by $\omega'$.  Note that, in the graph obtained from $G$ by contracting monochromatic components with respect to $\omega'$, the pivot vertex has eccentricity at most $r-1$, so by the inductive hypothesis we see that at most $(c-1)(r-1)$ further moves are required to flood $G$.  Hence the total number of moves required to flood $G$ is at most $(c-1) + (c-1)(r-1) = (c-1)r$, as required.
\end{proof}

In the free version, we can choose to play all moves at a vertex whose eccentricity is equal to the radius of the graph; this gives the following immediate corollary.

\begin{cor}
Let $G=(V,E)$ be a connected graph with radius $r$.  Then
$$M_c(G) \leq (c-1)r.$$
\label{radius-bound}
\end{cor}

We now argue that both of these bounds are tight.  To do this, we will make use of a specific family of trees: we define $T_{a,b}$ to be the tree obtained from the star $K_{1,a}$ by subdividing every edge exactly $b-1$ times (so $T_{a,b}$ is composed of $a$ paths on $b+1$ vertices, all having a common first vertex).  We begin with the fixed case.

\begin{lma}
Let $v$ be the unique vertex of $T_{(c-1)^r,r}$ with degree $(c-1)^r$.  Then
$$M_c^{(v)}(G) \geq (c-1)r.$$
\label{tight-trees-fixed}
\end{lma}
\begin{proof}
We define a $c$-colouring $\omega$ of $V(T_{(c-1)^r,r})$ such that $m^{(v)}(T_{(c-1)^r,r},\omega) \geq (c-1)r$.  We first set $\omega(v) = 1$.  Now let $\mathcal{S}_{c,r}$ be the set of all sequences of elements from $\{1,\ldots,c\}$ of length $r$ with the following properties:
\begin{enumerate}
\item the first element of the sequence is not 1, and
\item no two consecutive elements of the sequence are the same.
\end{enumerate}
Note that this definition implies that $|\mathcal{S}_{c,r}| = (c-1)^r$.  We will set our colouring $\omega$ to colour one of the paths in our tree with each $\sigma \in \mathcal{S}_{c,r}$: to colour a path with $\sigma$, we give the vertex adjacent to $v$ the colour that is the first element of $\sigma$, the next vertex the colour that is the second element, and so on.  Note that the conditions on elements of $\mathcal{S}_{c,r}$ ensure that this colouring $\omega$ is a proper colouring of $T_{c,r}$.

Let $S$ be any sequence of moves played at $v$ which floods $G$.  Note that there must be a colour $c_1$, other than $1$, that is none of the first $c-2$ moves of $S$.  We now define $c_i$ inductively: set $S_{i-1}$ to be the shortest initial segment of $S$ that is a supersequence of $c_1,\ldots,c_{i-1}$, and choose $c_i$ to be a colour, not equal to $c_{i-1}$, that does not appear in the first $c-2$ moves of the sequence $S$ after $S_{i-1}$ has been removed.  Observe that $c_1,\ldots,c_r$ is an element of $\mathcal{S}_{c,r}$, so there is some path in $T_{c,r}$ whose vertices (starting from the vertex adjacent to $v$) are coloured, in order, $c_1,\ldots,c_r$.  In order to flood this path, we must play $S_r$; but by construction, $|S_r| \geq (c-1)r$, so we must have $|S| \geq (c-1)r$, as claimed.
\end{proof}

We now generalise this argument to the free version; we use an almost identical construction, but with even more paths incident with the central vertex.

\begin{lma}
Let $T_{r(c-1)^{r+1},r}$ denote the tree obtained from $K_{1,r(c-1)^{r+1}}$ by subdividing each edge $r-1$ times.  Then
$$M_c(T_{r(c-1)^{r+1},r}) \geq (c-1)r.$$
\label{tight-trees}
\end{lma}
\begin{proof}
Once again, we set $v$ to be the vertex of $T_{r(c-1)^{r+1},r}$ with degree $r(c-1)^{r+1}$.  We define a colouring $\omega$ of $V(T_{r(c-1)^{r+1},r})$ by setting $\omega(v) =1$ and, for each $\sigma \in \mathcal{S}_{c,r}$, we colour $r(c-1)$ of the paths in our tree with $\sigma$.

If all moves are played at $v$, we can use exactly the same reasoning as in the proof of Lemma \ref{tight-trees-fixed}, so it remains only to argue that no sequence of fewer than $r(c-1)$ moves in which not all moves are played at $v$ can flood the tree.  Suppose that some number $\alpha$ of the moves are not played in monochromatic components containing $v$, with $1 \leq \alpha < (c-1)r$.  Note that any such move cannot change the colour of any vertex outside the path in which it is played.  Thus, even if $\alpha = (c-1)r - 1$, there must still be at least one path with each colouring from $\mathcal{S}_{c,r}$ whose colouring is not changed by any of these moves that is not played at $v$; at least $(c-1)r$ moves played at $v$ will be required to flood these remaining paths.
\end{proof}

\section{Graphs with high edge density}
\label{dense-graphs}

While we have seen that the upper bounds on $M_c(G)$ that we derived in the previous section are tight, it is natural to ask to whether they are only tight for graphs with few edges: intuitively, increasing the number of edges in the graph should make it easier to flood the graph.

One simple question we might ask is as follows: given $\delta \in (0,1)$, are the upper bounds in Propositions \ref{eccentricity-bound} and \ref{radius-bound} tight for any graph $G = (V,E)$ such that $|E|/|V|^2 \geq \delta$?  The answer to this question is yes: we can obtain a graph with arbitrarily high edge-density by adding a large clique to the constructions used in Lemmas \ref{tight-trees-fixed} and \ref{tight-trees}, every vertex of which is adjacent to the central vertex of the tree.  The addition of this clique does not change the radius of the graph, and no matter how we colour the clique the same number of moves will still be required to flood the whole graph.

However, this construction seems somewhat artificial: the value of $M_c(G)$ is determined by some small part of the graph whose edge-density remains small.  A more meaningful line of investigation is therefore to ask whether adding many edges (but no new vertices) to these constructions, in such a way that the radius remains unchanged, will significantly reduce $M_c(G)$.  

In this section, we provide a partial answer to this question.  If $v$ denotes the central vertex of some tree $T_{a,b}$, we define $V_i = \{u \in V(T_{a,b}): d(v,u) = i - 1\}$ for $1 \leq i \leq b+1$.  Then the set of edges we can add without changing the number of vertices in each $V_i$ is precisely
$$\{uw: u \in V_i, w \in V_j, |i-j| \leq 1\}.$$
Notice that the resulting graph is obtained from a blow-up of a path on $b+1$ vertices by making each vertex class (which is an independent set in the blow-up) into a clique.  Since adding edges cannot increase the number of moves required (by Corollary \ref{subgraph}), the following result tells us that $b + c - \frac{b-1}{c} - 2$ moves suffice to flood this graph in the free version.

\begin{prop}
Let $G = (V,E)$ be a blow-up of the path $P_t$ on $t$ vertices, and let $\omega: V(G) \rightarrow C$ be a colouring of $G$.  Suppose that $Q$ is a subpath of $G$ such that every vertex of $G$ has a neighbour on $Q$.  Then
$$m(G,\omega) \leq m(Q,\omega) + (c-1),$$
and in particular
$$m(G,\omega) \leq t-2 - \left \lceil \frac{t-2}{c} \right \rceil + (c-1).$$
\label{dominating-path}
\end{prop}
\begin{proof}
Note that, by Corollary \ref{subgraph}, $m_G(Q,\omega) \leq m_Q(Q,\omega)$, so it is possible to play at most $m_Q(Q,\omega)$ moves in $G$ to create a monochromatic component $A$ of some colour $d \in C$, where $A$ contains all of $Q$; by our assumptions on $Q$, every vertex in $G$ either belongs to $A$ or has a neighbour in $A$.  We can therefore flood the remainder of $G$ with at most $c-1$ further moves, changing the colour of $A$ repeatedly to give it every colour in $C \setminus \{d\}$.  This implies that
$$m(G,\omega) \leq m(Q,\omega) + (c-1),$$
as required.  The second part of the result then follows immediately from Proposition \ref{colour-bound}, together with the observation that a path containing precisely one vertex from every class except for the two end classes has the required property.
\end{proof}

The same argument gives an analogous result for the fixed version.

\begin{prop}
Let $G = (V,E)$ be a blow-up of the path $P_t$ on $t$ vertices, fix $v \in V$, and let $\omega:V(G) \rightarrow C$ be a colouring of $G$.  If $v$ belongs to one of the end classes of $G$, then
$$m^{(v)}(G,\omega) \leq (t-2) + (c-1) = t + c - 3;$$
if $v$ does not belong to either end class, then
$$m^{(v)}(G,\omega) \leq (t-3) + (c-1) = t + c - 4.$$
\label{dominating-path-fixed}
\end{prop}

The main result of this section is that, provided $b$ is sufficiently large compared with the number of colours, we can in fact improve on the simple bound of Proposition \ref{dominating-path} in the free version: the number of moves required to flood a blow-up of a long path is in fact the same as the number required, in the worst case, to flood a path of the same length.

\begin{thm}
Let $G$ be a blow-up of a path on $t$ vertices, and suppose that $c \geq 3$ and $t \geq 2c^{10}$.  Then 
$$M_c(G) = t - \left \lceil \frac{t}{c} \right \rceil.$$
\label{blowup-thm}
\end{thm}

We have made no attempt to optimise the dependence of $t$ on $c$ in the statement of Theorem \ref{blowup-thm}, and indeed conjecture that the result is true for much smaller values of $t$.  However, it is clear that some dependence on $c$ is necessary, as it follows from Proposition \ref{c-col} that if $G$ is a blow-up of a path on $c$ vertices in which every vertex class has size at least two and, for $1 \leq i \leq c$, $\omega$ assigns colour $i$ to both vertices corresponding to the $i^{th}$ vertex of the path.

We also remark that we cannot improve on the bound of Proposition \ref{dominating-path-fixed} in the fixed case: suppose that the vertex classes are $V_1,\ldots,V_t$, with the pivot $v \in V_1$, and that the classes $V_1,\ldots,V_{t-1}$ are alternately coloured with colours $1$ and $2$, while $V_t$ contains a vertex of every colour.  It is easy to verify that we require $t+c-3$ moves in this case.

The remainder of the section is devoted to the proof of Theorem \ref{blowup-thm}.  In Section \ref{rainbow-col} we consider a generalisation of rainbow colourings to blow-ups of paths and obtain an upper bound on the number of moves required in this case; in Section \ref{path-col} we generalise our results to colourings in which all vertices ``blown up'' from a single vertex receive the same colour; finally, in Section \ref{all-col}, we deal with arbitrary colourings.

First, we observe that the upper bound in Theorem \ref{blowup-thm} is optimal.  Recall from the definition of a blow-up of a graph that, if $G$ is a blow-up of a path on $t$ vertices, the vertices of $G$ can be partitioned into vertex-classes $V_1,\ldots,V_t$ such that each class $V_i$ is an independent set in $G$ and $uw$ is an edge in $G$ if and only if $u \in V_i$ and $w \in V_j$ where $|i - j| = 1$.

\begin{lma}
Let $G$ be a blow-up of a path on $t$ vertices.  Then
$$M_c(G) \geq t - \left \lceil \frac{t}{c} \right \rceil.$$
\label{blowup-lb}
\end{lma}
\begin{proof}
We define a colouring $\omega: V(G) \rightarrow \{0,\ldots,c-1\}$ by setting $\omega(v) = i \mod c$, where $v \in V_i$, for $1 \leq i \leq t$.  Let $G'$ be the graph obtained from $G$ by adding all edges within each $V_i$; by Corollary \ref{subgraph} this cannot increase the number of moves required to flood the graph.  Moreover, it is clear that $G'$ with colouring $\omega$ is equivalent (contracting monochromatic components) to a path on $t$ vertices with a $C$-rainbow colouring.  We therefore know from Corollary \ref{path-result} that $m(G',\omega) \geq t - \left \lceil \frac{t}{c} \right \rceil$, so we also have $m(G,\omega) \geq t - \left \lceil \frac{t}{c} \right \rceil$, as required.
\end{proof}

\subsection{Rainbow colourings}
\label{rainbow-col}

We begin by defining an important restricted family of colourings for graphs that are blow-ups of paths.

\begin{adef}
Let $G=(V,E)$ be a blow-up of the path $P_t$ on $t$ vertices, and let $C = \{d_0,\ldots,d_{c-1}\}$ be a set of colours.  We say that the colouring $\omega: V \rightarrow C$ is a \emph{path colouring} of $G$ if there exists a function $f: \{1,\ldots,t\} \rightarrow C$ such that, for each $1 \leq i \leq t$, we have $\omega(v) = f(i)$ for every $v \in V_i$.
\end{adef}

Using this definition, we extend our definition of $C$-rainbow colourings to blow-ups of paths, to define a subfamily of path colourings.

\begin{adef}
Let $G=(V,E)$ be a blow-up of the path $P_t$ on $t$ vertices.  Suppose that $C=\{d_0,\ldots,d_{c-1}\}$ is a set of colours, and $\omega: V \rightarrow C$ is a path colouring of $V$.  The colouring $\omega$ is said to be a $C$-\emph{rainbow} colouring of $G$ if the corresponding colouring of $P_t$ is a $C$-rainbow colouring of the path.
\end{adef}

As for paths, we say that $\omega$ is a \emph{rainbow} colouring if it is a $C$-rainbow colouring for some colour-set $C$.  Note that the colouring used in the proof of Lemma \ref{blowup-lb} is a rainbow colouring.

We now prove that Theorem \ref{blowup-thm} holds if we restrict our attention only to rainbow colourings.

\begin{lma}
Let $G$ be a blow-up of the path $P_t$, and let $\omega$ be a $C$-rainbow colouring of $G$.  Then, if $t \geq c + 2$, we have
$$m(G,\omega) \leq t - \left \lceil \frac{t}{c} \right \rceil.$$
\label{rainbow-blowup}
\end{lma}
\begin{proof}
We prove the result by induction on $t$.  We begin by considering several base cases, which together cover the situation in which $c+2 \leq t \leq 3c + 1$.  As usual, we will denote by $V_1,\ldots,V_t$ the vertex-classes of $G$, and we may assume without loss of generality that $V_i$ receives colour $i \mod c$ under $\omega$.

For the first of these cases, suppose that $c + 2 \leq t \leq 2c$.  We describe a strategy to flood $G$ with $t - \left \lceil \frac{t}{c} \right \rceil = t - 2$ moves; this strategy is illustrated in Figures \ref{base-case-1} and \ref{base-case-1-redux}.  First, we play $c-1$ moves at some vertex $v \in V_2$, giving this vertex colours $3,4,\ldots,c-1,0,1$ in turn; this will create a monochromatic component containing $v$ and all of $V_1 \cup V_3 \cup \cdots \cup V_{c+1}$.  Note that the only vertices of $V_1 \cup \cdots \cup V_{c+1}$ that do not belong to this monochromatic component are in $V_2$ and have colour $2$.  Now we change the colour of this component to take colours $2,\ldots, t \mod c$ in turn; this will extend our monochromatic component to contain all of $V_{c+2} \cup \cdots \cup V_{t}$ and, as the component takes colour $2$ at the start of this subsequence of moves, all remaining vertices of $V_2$ will also be flooded.  Thus we have described a sequence of $c - 1 + t - (c+1) = t - 2$ moves which floods $G$, as required.

\begin{figure}

\centering
\begin{tikzpicture}[scale=.5]
\draw (0,18) ellipse (.8 and 1.5); \draw node at (0,18) {\scriptsize $1$}; \draw node at (0,20) {\scriptsize $V_1$};
\draw (.8,19)--(2.2,19)--(.8,18.5)--(2.2,18.5)--(.8,18)--(2.2,18)--(.8,17.5)--(2.2,17.5)--(.8,17)--(2.2,17);
\draw (2.2,19)--(.8,19)--(2.2,18.5)--(.8,18.5)--(2.2,18)--(.8,18)--(2.2,17.5)--(.8,17.5)--(2.2,17)--(.8,17);

\draw (3,18) ellipse (.8 and 1.5); \draw node at (3,19) {\scriptsize $2$};  \draw node at (3,20) {\scriptsize $V_2$};
\draw[fill] (3,17.5) circle(0.1); \draw node at (3,17) {\scriptsize $v$};
\draw (3.8,19)--(5.2,19)--(3.8,18.5)--(5.2,18.5)--(3.8,18)--(5.2,18)--(3.8,17.5)--(5.2,17.5)--(3.8,17)--(5.2,17);
\draw (5.2,19)--(3.8,19)--(5.2,18.5)--(3.8,18.5)--(5.2,18)--(3.8,18)--(5.2,17.5)--(3.8,17.5)--(5.2,17)--(3.8,17);

\draw (6.8,19)--(8.2,19)--(6.8,18.5)--(8.2,18.5)--(6.8,18)--(8.2,18)--(6.8,17.5)--(8.2,17.5)--(6.8,17)--(8.2,17);
\draw (8.2,19)--(6.8,19)--(8.2,18.5)--(6.8,18.5)--(8.2,18)--(6.8,18)--(8.2,17.5)--(6.8,17.5)--(8.2,17)--(6.8,17);

\draw node at (6,18) {\scriptsize $\dots$};

\draw (9,18) ellipse (.8 and 1.5); \draw node at (9,18) {\scriptsize $c-1$}; \draw node at (9,20) {\scriptsize $V_{c-1}$};
\draw (9.8,19)--(11.2,19)--(9.8,18.5)--(11.2,18.5)--(9.8,18)--(11.2,18)--(9.8,17.5)--(11.2,17.5)--(9.8,17)--(11.2,17);
\draw (11.2,19)--(9.8,19)--(11.2,18.5)--(9.8,18.5)--(11.2,18)--(9.8,18)--(11.2,17.5)--(9.8,17.5)--(11.2,17)--(9.8,17);

\draw (12,18) ellipse (.8 and 1.5); \draw node at (12,18) {\scriptsize $0$}; \draw node at (12,20) {\scriptsize $V_c$};
\draw (12.8,19)--(14.2,19)--(12.8,18.5)--(14.2,18.5)--(12.8,18)--(14.2,18)--(12.8,17.5)--(14.2,17.5)--(12.8,17)--(14.2,17);
\draw (14.2,19)--(12.8,19)--(14.2,18.5)--(12.8,18.5)--(14.2,18)--(12.8,18)--(14.2,17.5)--(12.8,17.5)--(14.2,17)--(12.8,17);

\draw (15,18) ellipse (.8 and 1.5); \draw node at (15,18) {\scriptsize $1$}; \draw node at (15,20) {\scriptsize $V_{c+1}$};
\draw (15.8,19)--(17.2,19)--(15.8,18.5)--(17.2,18.5)--(15.8,18)--(17.2,18)--(15.8,17.5)--(17.2,17.5)--(15.8,17)--(17.2,17);
\draw (17.2,19)--(15.8,19)--(17.2,18.5)--(15.8,18.5)--(17.2,18)--(15.8,18)--(17.2,17.5)--(15.8,17.5)--(17.2,17)--(15.8,17);

\draw (18,18) ellipse (.8 and 1.5); \draw node at (18,18) {\scriptsize $2$}; \draw node at (18,20) {\scriptsize $V_{c+2}$};
\draw (18.8,19)--(20.2,19)--(18.8,18.5)--(20.2,18.5)--(18.8,18)--(20.2,18)--(18.8,17.5)--(20.2,17.5)--(18.8,17)--(20.2,17);
\draw (20.2,19)--(18.8,19)--(20.2,18.5)--(18.8,18.5)--(20.2,18)--(18.8,18)--(20.2,17.5)--(18.8,17.5)--(20.2,17)--(18.8,17);

\draw (21,18) ellipse (.8 and 1.5); \draw node at (21,18) {\scriptsize $3$}; \draw node at (21,20) {\scriptsize $V_{c+3}$};
\draw (21.8,19)--(23.2,19)--(21.8,18.5)--(23.2,18.5)--(21.8,18)--(23.2,18)--(21.8,17.5)--(23.2,17.5)--(21.8,17)--(23.2,17);
\draw (23.2,19)--(21.8,19)--(23.2,18.5)--(21.8,18.5)--(23.2,18)--(21.8,18)--(23.2,17.5)--(21.8,17.5)--(23.2,17)--(21.8,17);

\draw node at (24,18) {\scriptsize $\dots$};

\draw (24.8,19)--(26.2,19)--(24.8,18.5)--(26.2,18.5)--(24.8,18)--(26.2,18)--(24.8,17.5)--(26.2,17.5)--(24.8,17)--(26.2,17);
\draw (26.2,19)--(24.8,19)--(26.2,18.5)--(24.8,18.5)--(26.2,18)--(24.8,18)--(26.2,17.5)--(24.8,17.5)--(26.2,17)--(24.8,17);
\draw (27,18) ellipse (.8 and 1.5); \draw node at (27,18) {\scriptsize $t_c$}; \draw node at (27,20) {\scriptsize $V_t$};
%%%%%%%%%%%%%%%%%%%%%%%%%%%%%%%%%%%%%%%%%%%%%%%%%%%%%%%%%%%%%%%%%%%%%%%%%%%%%
\path[->, font=\scriptsize] (7.5,16) edge node[auto] {$c-1$ moves} (7.5,14);

\draw (0,12) ellipse (.8 and 1.5); \draw node at (0,12) {\scriptsize $1$};
\draw (.8,13)--(2.2,13)--(.8,12.5)--(2.2,12.5)--(.8,12)--(2.2,12)--(.8,11.5)--(2.2,11.5);
\draw (2.2,13)--(.8,13)--(2.2,12.5)--(.8,12.5)--(2.2,12)--(.8,12)--(2.2,11.5)--(.8,11.5);

\draw (3,12.5) ellipse (.8 and 1); \draw node at (3,12.5) {\scriptsize $2$};
\draw (3.8,13)--(5.2,13)--(3.8,12.5)--(5.2,12.5)--(3.8,12)--(5.2,12)--(3.8,11.5)--(5.2,11.5);
\draw (5.2,13)--(3.8,13)--(5.2,12.5)--(3.8,12.5)--(5.2,12)--(3.8,12)--(5.2,11.5)--(3.8,11.5);

\draw[fill] (3,10.5) circle(0.1); \draw node at (3,10) {\scriptsize $v$ (colour 1)};
\draw (3,10.5)--(.8,13);\draw (3,10.5)--(5.2,13);\draw (3,10.5)--(.8,12.5);\draw (3,10.5)--(5.2,12.5);\draw (3,10.5)--(.8,12);\draw (3,10.5)--(5.2,12);\draw (3,10.5)--(.8,11.5);\draw (3,10.5)--(5.2,11.5);\draw (3,10.5)--(.8,11);\draw (3,10.5)--(5.2,11);

\draw (6.8,13)--(8.2,13)--(6.8,12.5)--(8.2,12.5)--(6.8,12)--(8.2,12)--(6.8,11.5)--(8.2,11.5)--(6.8,11)--(8.2,11);
\draw (8.2,13)--(6.8,13)--(8.2,12.5)--(6.8,12.5)--(8.2,12)--(6.8,12)--(8.2,11.5)--(6.8,11.5)--(8.2,11)--(6.8,11);

\draw node at (6,12) {\scriptsize $\dots$};

\draw (9,12) ellipse (.8 and 1.5); \draw node at (9,12) {\scriptsize $1$};
\draw (9.8,13)--(11.2,13)--(9.8,12.5)--(11.2,12.5)--(9.8,12)--(11.2,12)--(9.8,11.5)--(11.2,11.5)--(9.8,11)--(11.2,11);
\draw (11.2,13)--(9.8,13)--(11.2,12.5)--(9.8,12.5)--(11.2,12)--(9.8,12)--(11.2,11.5)--(9.8,11.5)--(11.2,11)--(9.8,11);

\draw (12,12) ellipse (.8 and 1.5); \draw node at (12,12) {\scriptsize $1$};
\draw (12.8,13)--(14.2,13)--(12.8,12.5)--(14.2,12.5)--(12.8,12)--(14.2,12)--(12.8,11.5)--(14.2,11.5)--(12.8,11)--(14.2,11);
\draw (14.2,13)--(12.8,13)--(14.2,12.5)--(12.8,12.5)--(14.2,12)--(12.8,12)--(14.2,11.5)--(12.8,11.5)--(14.2,11)--(12.8,11);

\draw (15,12) ellipse (.8 and 1.5); \draw node at (15,12) {\scriptsize $1$};
\draw (15.8,13)--(17.2,13)--(15.8,12.5)--(17.2,12.5)--(15.8,12)--(17.2,12)--(15.8,11.5)--(17.2,11.5)--(15.8,11)--(17.2,11);
\draw (17.2,13)--(15.8,13)--(17.2,12.5)--(15.8,12.5)--(17.2,12)--(15.8,12)--(17.2,11.5)--(15.8,11.5)--(17.2,11)--(15.8,11);

\draw (18,12) ellipse (.8 and 1.5); \draw node at (18,12) {\scriptsize $2$};
\draw (18.8,13)--(20.2,13)--(18.8,12.5)--(20.2,12.5)--(18.8,12)--(20.2,12)--(18.8,11.5)--(20.2,11.5)--(18.8,11)--(20.2,11);
\draw (20.2,13)--(18.8,13)--(20.2,12.5)--(18.8,12.5)--(20.2,12)--(18.8,12)--(20.2,11.5)--(18.8,11.5)--(20.2,11)--(18.8,11);

\draw (21,12) ellipse (.8 and 1.5); \draw node at (21,12) {\scriptsize $3$};
\draw (21.8,13)--(23.2,13)--(21.8,12.5)--(23.2,12.5)--(21.8,12)--(23.2,12)--(21.8,11.5)--(23.2,11.5)--(21.8,11)--(23.2,11);
\draw (23.2,13)--(21.8,13)--(23.2,12.5)--(21.8,12.5)--(23.2,12)--(21.8,12)--(23.2,11.5)--(21.8,11.5)--(23.2,11)--(21.8,11);

\draw node at (24,12) {\scriptsize $\dots$};

\draw (24.8,13)--(26.2,13)--(24.8,12.5)--(26.2,12.5)--(24.8,12)--(26.2,12)--(24.8,11.5)--(26.2,11.5)--(24.8,11)--(26.2,11);
\draw (26.2,13)--(24.8,13)--(26.2,12.5)--(24.8,12.5)--(26.2,12)--(24.8,12)--(26.2,11.5)--(24.8,11.5)--(26.2,11)--(24.8,11);
\draw (27,12) ellipse (.8 and 1.5); \draw node at (27,12) {\scriptsize $t_c$};
%%%%%%%%%%%%%%%%%%%%%%%%%%%%%%%%%%%%%%%%%%%%%%%%%%%%%%%%%%%%%%%%%%%%%%%%%%%%%%%%%%%
\path[->, font=\scriptsize] (13.5,10) edge node[auto] {$1$ move} (13.5,8); 

\draw (0,6) ellipse (.8 and 1.5); \draw node at (0,6) {\scriptsize $2$};
\draw (.8,7)--(2.2,7)--(.8,6.5)--(2.2,6.5)--(.8,6)--(2.2,6)--(.8,5.5)--(2.2,5.5)--(.8,5)--(2.2,5);
\draw (2.2,7)--(.8,7)--(2.2,6.5)--(.8,6.5)--(2.2,6)--(.8,6)--(2.2,5.5)--(.8,5.5)--(2.2,5)--(.8,5);

\draw (3,6) ellipse (.8 and 1.5); \draw node at (3,7) {\scriptsize $2$};
\draw[fill] (3,5.5) circle(0.1); \draw node at (3,5) {\scriptsize $v$};
\draw (3.8,7)--(5.2,7)--(3.8,6.5)--(5.2,6.5)--(3.8,6)--(5.2,6)--(3.8,5.5)--(5.2,5.5)--(3.8,5)--(5.2,5);
\draw (5.2,7)--(3.8,7)--(5.2,6.5)--(3.8,6.5)--(5.2,6)--(3.8,6)--(5.2,5.5)--(3.8,5.5)--(5.2,5)--(3.8,5);

\draw (6.8,7)--(8.2,7)--(6.8,6.5)--(8.2,6.5)--(6.8,6)--(8.2,6)--(6.8,5.5)--(8.2,5.5)--(6.8,5)--(8.2,5);
\draw (8.2,7)--(6.8,7)--(8.2,6.5)--(6.8,6.5)--(8.2,6)--(6.8,6)--(8.2,5.5)--(6.8,5.5)--(8.2,5)--(6.8,5);

\draw node at (6,6) {\scriptsize $\dots$};

\draw (9,6) ellipse (.8 and 1.5); \draw node at (9,6) {\scriptsize $2$};
\draw (9.8,7)--(11.2,7)--(9.8,6.5)--(11.2,6.5)--(9.8,6)--(11.2,6)--(9.8,5.5)--(11.2,5.5)--(9.8,5)--(11.2,5);
\draw (11.2,7)--(9.8,7)--(11.2,6.5)--(9.8,6.5)--(11.2,6)--(9.8,6)--(11.2,5.5)--(9.8,5.5)--(11.2,5)--(9.8,5);

\draw (12,6) ellipse (.8 and 1.5); \draw node at (12,6) {\scriptsize $2$};
\draw (12.8,7)--(14.2,7)--(12.8,6.5)--(14.2,6.5)--(12.8,6)--(14.2,6)--(12.8,5.5)--(14.2,5.5)--(12.8,5)--(14.2,5);
\draw (14.2,7)--(12.8,7)--(14.2,6.5)--(12.8,6.5)--(14.2,6)--(12.8,6)--(14.2,5.5)--(12.8,5.5)--(14.2,5)--(12.8,5);

\draw (15,6) ellipse (.8 and 1.5); \draw node at (15,6) {\scriptsize $2$};
\draw (15.8,7)--(17.2,7)--(15.8,6.5)--(17.2,6.5)--(15.8,6)--(17.2,6)--(15.8,5.5)--(17.2,5.5)--(15.8,5)--(17.2,5);
\draw (17.2,7)--(15.8,7)--(17.2,6.5)--(15.8,6.5)--(17.2,6)--(15.8,6)--(17.2,5.5)--(15.8,5.5)--(17.2,5)--(15.8,5);

\draw (18,6) ellipse (.8 and 1.5); \draw node at (18,6) {\scriptsize $2$};
\draw (18.8,7)--(20.2,7)--(18.8,6.5)--(20.2,6.5)--(18.8,6)--(20.2,6)--(18.8,5.5)--(20.2,5.5)--(18.8,5)--(20.2,5);
\draw (20.2,7)--(18.8,7)--(20.2,6.5)--(18.8,6.5)--(20.2,6)--(18.8,6)--(20.2,5.5)--(18.8,5.5)--(20.2,5)--(18.8,5);

\draw (21,6) ellipse (.8 and 1.5); \draw node at (21,6) {\scriptsize $3$};
\draw (21.8,7)--(23.2,7)--(21.8,6.5)--(23.2,6.5)--(21.8,6)--(23.2,6)--(21.8,5.5)--(23.2,5.5)--(21.8,5)--(23.2,5);
\draw (23.2,7)--(21.8,7)--(23.2,6.5)--(21.8,6.5)--(23.2,6)--(21.8,6)--(23.2,5.5)--(21.8,5.5)--(23.2,5)--(21.8,5);

\draw node at (24,6) {\scriptsize $\dots$};

\draw (24.8,7)--(26.2,7)--(24.8,6.5)--(26.2,6.5)--(24.8,6)--(26.2,6)--(24.8,5.5)--(26.2,5.5)--(24.8,5)--(26.2,5);
\draw (26.2,7)--(24.8,7)--(26.2,6.5)--(24.8,6.5)--(26.2,6)--(24.8,6)--(26.2,5.5)--(24.8,5.5)--(26.2,5)--(24.8,5);
\draw (27,6) ellipse (.8 and 1.5); \draw node at (27,6) {\scriptsize $t_c$};
%%%%%%%%%%%%%%%%%%%%%%%%%%%%%%%%%%%%%%%%%%%%%%%%%%%%%%%%%%%%%%%%%%%%%%%%%%%%%%%%%%%
\path[->, font=\scriptsize] (19.5,4) edge node[auto] {$t-(c+2)$ moves} (19.5,2); 

\draw (0,0) ellipse (.8 and 1.5); \draw node at (0,0) {\scriptsize $t_c$};
\draw (.8,1)--(2.2,1)--(.8,.5)--(2.2,.5)--(.8,0)--(2.2,0)--(.8,-.5)--(2.2,-.5)--(.8,-1)--(2.2,-1);
\draw (2.2,1)--(.8,1)--(2.2,.5)--(.8,.5)--(2.2,0)--(.8,0)--(2.2,-.5)--(.8,-.5)--(2.2,-1)--(.8,-1);

\draw (3,0) ellipse (.8 and 1.5); \draw node at (3,0) {\scriptsize $t_c$};
\draw (3.8,1)--(5.2,1)--(3.8,.5)--(5.2,.5)--(3.8,0)--(5.2,0)--(3.8,-.5)--(5.2,-.5)--(3.8,-1)--(5.2,-1);
\draw (5.2,1)--(3.8,1)--(5.2,.5)--(3.8,.5)--(5.2,0)--(3.8,0)--(5.2,-.5)--(3.8,-.5)--(5.2,-1)--(3.8,-1);

\draw (6.8,1)--(8.2,1)--(6.8,.5)--(8.2,.5)--(6.8,0)--(8.2,0)--(6.8,-.5)--(8.2,-.5)--(6.8,-1)--(8.2,-1);
\draw (8.2,1)--(6.8,1)--(8.2,.5)--(6.8,.5)--(8.2,0)--(6.8,0)--(8.2,-.5)--(6.8,-.5)--(8.2,-1)--(6.8,-1);

\draw node at (6,0) {\scriptsize $\dots$};

\draw (9,0) ellipse (.8 and 1.5); \draw node at (9,0) {\scriptsize $t_c$};
\draw (9.8,1)--(11.2,1)--(9.8,.5)--(11.2,.5)--(9.8,0)--(11.2,0)--(9.8,-.5)--(11.2,-.5)--(9.8,-1)--(11.2,-1);
\draw (11.2,1)--(9.8,1)--(11.2,.5)--(9.8,.5)--(11.2,0)--(9.8,0)--(11.2,-.5)--(9.8,-.5)--(11.2,-1)--(9.8,-1);

\draw (12,0) ellipse (.8 and 1.5); \draw node at (12,0) {\scriptsize $t_c$};
\draw (12.8,1)--(14.2,1)--(12.8,.5)--(14.2,.5)--(12.8,0)--(14.2,0)--(12.8,-.5)--(14.2,-.5)--(12.8,-1)--(14.2,-1);
\draw (14.2,1)--(12.8,1)--(14.2,.5)--(12.8,.5)--(14.2,0)--(12.8,0)--(14.2,-.5)--(12.8,-.5)--(14.2,-1)--(12.8,-1);

\draw (15,0) ellipse (.8 and 1.5); \draw node at (15,0) {\scriptsize $t_c$};
\draw (15.8,1)--(17.2,1)--(15.8,.5)--(17.2,.5)--(15.8,0)--(17.2,0)--(15.8,-.5)--(17.2,-.5)--(15.8,-1)--(17.2,-1);
\draw (17.2,1)--(15.8,1)--(17.2,.5)--(15.8,.5)--(17.2,0)--(15.8,0)--(17.2,-.5)--(15.8,-.5)--(17.2,-1)--(15.8,-1);

\draw (18,0) ellipse (.8 and 1.5); \draw node at (18,0) {\scriptsize $t_c$};
\draw (18.8,1)--(20.2,1)--(18.8,.5)--(20.2,.5)--(18.8,0)--(20.2,0)--(18.8,-.5)--(20.2,-.5)--(18.8,-1)--(20.2,-1);
\draw (20.2,1)--(18.8,1)--(20.2,.5)--(18.8,.5)--(20.2,0)--(18.8,0)--(20.2,-.5)--(18.8,-.5)--(20.2,-1)--(18.8,-1);

\draw (21,0) ellipse (.8 and 1.5); \draw node at (21,0) {\scriptsize $t_c$};
\draw (21.8,1)--(23.2,1)--(21.8,.5)--(23.2,.5)--(21.8,0)--(23.2,0)--(21.8,-.5)--(23.2,-.5)--(21.8,-1)--(23.2,-1);
\draw (23.2,1)--(21.8,1)--(23.2,.5)--(21.8,.5)--(23.2,0)--(21.8,0)--(23.2,-.5)--(21.8,-.5)--(23.2,-1)--(21.8,-1);

\draw node at (24,0) {\scriptsize $\dots$};

\draw (24.8,1)--(26.2,1)--(24.8,.5)--(26.2,.5)--(24.8,0)--(26.2,0)--(24.8,-.5)--(26.2,-.5)--(24.8,-1)--(26.2,-1);
\draw (26.2,1)--(24.8,1)--(26.2,.5)--(24.8,.5)--(26.2,0)--(24.8,0)--(26.2,-.5)--(24.8,-.5)--(26.2,-1)--(24.8,-1);
\draw (27,0) ellipse (.8 and 1.5); \draw node at (27,0) {\scriptsize $t_c$};
\end{tikzpicture}
\caption{The first base case for Lemma \ref{rainbow-blowup}: detailed version.  We write $t_c$ for $t \mod c$.}
\label{base-case-1}
\end{figure}
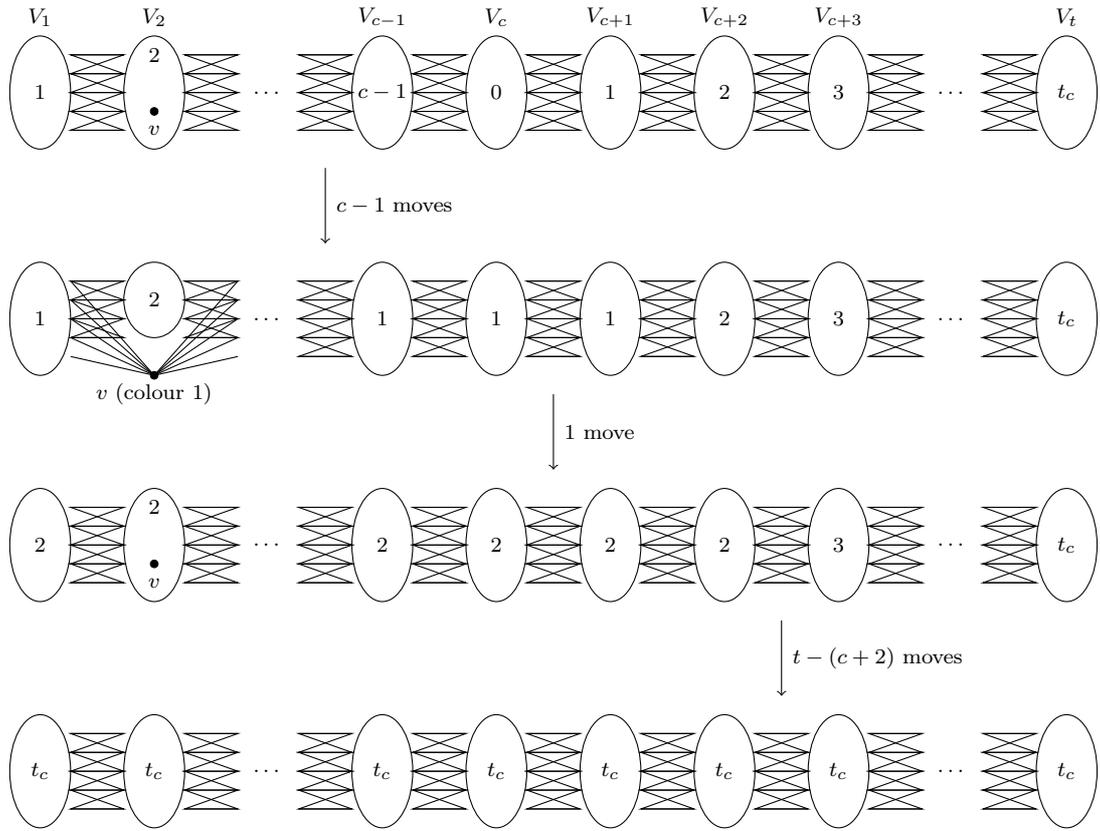
%Figure~\ref{base-case-1} illustrates the first base case in detail, for the following cases, we will only give a reduced schematic. The reduced version of the first case is given in the following figure.
%
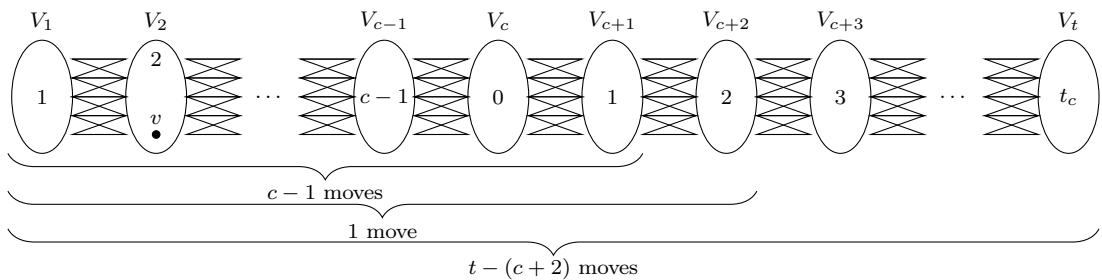
\begin{figure}
\centering
\begin{tikzpicture}[scale=.5]
\draw (0,0) ellipse (.8 and 1.5); \draw node at (0,0) {\scriptsize $1$}; \draw node at (0,2) {\scriptsize $V_1$};
\draw (.8,1)--(2.2,1)--(.8,.5)--(2.2,.5)--(.8,0)--(2.2,0)--(.8,-.5)--(2.2,-.5)--(.8,-1)--(2.2,-1);
\draw (2.2,1)--(.8,1)--(2.2,.5)--(.8,.5)--(2.2,0)--(.8,0)--(2.2,-.5)--(.8,-.5)--(2.2,-1)--(.8,-1);

\draw (3,0) ellipse (.8 and 1.5); \draw node at (3,1) {\scriptsize $2$};\draw node at (3,2) {\scriptsize $V_2$};
\draw[fill] (3,-1) circle (0.1); \draw[above] node at (3,-1){\scriptsize $v$};
\draw (3.8,1)--(5.2,1)--(3.8,.5)--(5.2,.5)--(3.8,0)--(5.2,0)--(3.8,-.5)--(5.2,-.5)--(3.8,-1)--(5.2,-1);
\draw (5.2,1)--(3.8,1)--(5.2,.5)--(3.8,.5)--(5.2,0)--(3.8,0)--(5.2,-.5)--(3.8,-.5)--(5.2,-1)--(3.8,-1);

\draw (6.8,1)--(8.2,1)--(6.8,.5)--(8.2,.5)--(6.8,0)--(8.2,0)--(6.8,-.5)--(8.2,-.5)--(6.8,-1)--(8.2,-1);
\draw (8.2,1)--(6.8,1)--(8.2,.5)--(6.8,.5)--(8.2,0)--(6.8,0)--(8.2,-.5)--(6.8,-.5)--(8.2,-1)--(6.8,-1);

\draw node at (6,0) {\scriptsize $\dots$};

\draw (9,0) ellipse (.8 and 1.5); \draw node at (9,0) {\scriptsize $c-1$}; \draw node at (9,2) {\scriptsize $V_{c-1}$};
\draw (9.8,1)--(11.2,1)--(9.8,.5)--(11.2,.5)--(9.8,0)--(11.2,0)--(9.8,-.5)--(11.2,-.5)--(9.8,-1)--(11.2,-1);
\draw (11.2,1)--(9.8,1)--(11.2,.5)--(9.8,.5)--(11.2,0)--(9.8,0)--(11.2,-.5)--(9.8,-.5)--(11.2,-1)--(9.8,-1);

\draw (12,0) ellipse (.8 and 1.5); \draw node at (12,0) {\scriptsize $0$}; \draw node at (12,2) {\scriptsize $V_{c}$};
\draw (12.8,1)--(14.2,1)--(12.8,.5)--(14.2,.5)--(12.8,0)--(14.2,0)--(12.8,-.5)--(14.2,-.5)--(12.8,-1)--(14.2,-1);
\draw (14.2,1)--(12.8,1)--(14.2,.5)--(12.8,.5)--(14.2,0)--(12.8,0)--(14.2,-.5)--(12.8,-.5)--(14.2,-1)--(12.8,-1);

\draw (15,0) ellipse (.8 and 1.5); \draw node at (15,0) {\scriptsize $1$}; \draw node at (15,2) {\scriptsize $V_{c+1}$};
\draw (15.8,1)--(17.2,1)--(15.8,.5)--(17.2,.5)--(15.8,0)--(17.2,0)--(15.8,-.5)--(17.2,-.5)--(15.8,-1)--(17.2,-1);
\draw (17.2,1)--(15.8,1)--(17.2,.5)--(15.8,.5)--(17.2,0)--(15.8,0)--(17.2,-.5)--(15.8,-.5)--(17.2,-1)--(15.8,-1);

\draw (18,0) ellipse (.8 and 1.5); \draw node at (18,0) {\scriptsize $2$}; \draw node at (18,2) {\scriptsize $V_{c+2}$};
\draw (18.8,1)--(20.2,1)--(18.8,.5)--(20.2,.5)--(18.8,0)--(20.2,0)--(18.8,-.5)--(20.2,-.5)--(18.8,-1)--(20.2,-1);
\draw (20.2,1)--(18.8,1)--(20.2,.5)--(18.8,.5)--(20.2,0)--(18.8,0)--(20.2,-.5)--(18.8,-.5)--(20.2,-1)--(18.8,-1);

\draw (21,0) ellipse (.8 and 1.5); \draw node at (21,0) {\scriptsize $3$};  \draw node at (21,2) {\scriptsize $V_{c+3}$};
\draw (21.8,1)--(23.2,1)--(21.8,.5)--(23.2,.5)--(21.8,0)--(23.2,0)--(21.8,-.5)--(23.2,-.5)--(21.8,-1)--(23.2,-1);
\draw (23.2,1)--(21.8,1)--(23.2,.5)--(21.8,.5)--(23.2,0)--(21.8,0)--(23.2,-.5)--(21.8,-.5)--(23.2,-1)--(21.8,-1);

\draw node at (24,0) {\scriptsize $\dots$};

\draw (24.8,1)--(26.2,1)--(24.8,.5)--(26.2,.5)--(24.8,0)--(26.2,0)--(24.8,-.5)--(26.2,-.5)--(24.8,-1)--(26.2,-1);
\draw (26.2,1)--(24.8,1)--(26.2,.5)--(24.8,.5)--(26.2,0)--(24.8,0)--(26.2,-.5)--(24.8,-.5)--(26.2,-1)--(24.8,-1);
\draw (27,0) ellipse (.8 and 1.5); \draw node at (27,0) {\scriptsize $t_c$};  \draw node at (27,2) {\scriptsize $V_{t}$};

\draw [decorate,decoration={brace,amplitude=10, mirror},xshift=0,yshift=0] (-0.9,-1.5) -- (15.8,-1.5) node [black,midway,yshift=-15] {\scriptsize $c-1$ moves};
\draw [decorate,decoration={brace,amplitude=10, mirror},xshift=0,yshift=0] (-0.9,-2.5) -- (18.8,-2.5) node [black,midway,yshift=-15] {\scriptsize $1$ move};
\draw [decorate,decoration={brace,amplitude=10, mirror},xshift=0,yshift=0] (-0.9,-3.5) -- (27.8,-3.5) node [black,midway,yshift=-15] {\scriptsize $t-(c+2)$ moves};
\end{tikzpicture}
\caption{A reduced schematic of the first base case of Lemma \ref{rainbow-blowup}.  We write $t_c$ for $t \mod c$.}
\label{base-case-1-redux}
\end{figure}

For the second base case, suppose that $2c + 1 \leq t \leq 3c$.  In this case we play $t - \left \lceil \frac{t}{c} \right \rceil = t - 3$ moves, all at a vertex $v \in V_{c+1}$, as illustrated in Figure \ref{base-case-2}.  We first change the colour of $v$ to take colours $2,3,\ldots,0$ in turn; this creates a monochromatic component containing $v$ and all of $V_c \cup V_{c+2} \cup \cdots \cup V_{2c}$.  Next we give $v$ colours $c-1,c-2,\ldots,2$ in turn; at this point there is a monochromatic component containing all of $V_2 \cup \cdots \cup V_{2c}$ except for some vertices in $V_{c+1}$ of colour $1$.  Playing one further move to give this component colour $1$ therefore creates a monochromatic component containing all of $V_1 \cup \cdots \cup V_{2c+1}$.  To flood the remainder of $G$, we give this component colours $2,\ldots,t \mod c$ in turn.  This strategy allows us to flood $G$ in a total of $c - 1 + c - 2 + 1 + t - (2c + 1) = t - 3$ moves, as required.

\begin{figure}[h]
\centering
\begin{tikzpicture}[scale=.45]
\draw (0,0) ellipse (.8 and 1.5); \draw node at (0,0) {\scriptsize $1$};  \draw node at (0,2) {\scriptsize $V_{1}$};
\draw (.66,1)--(1.33,1)--(.72,.5)--(1.28,.5)--(.8,0)--(1.2,0)--(.72,-.5)--(1.25,-.5)--(.66,-1)--(1.33,-1);
\draw (1.33,1)--(.66,1)--(1.28,.5)--(.72,.5)--(1.2,0)--(.8,0)--(1.28,-.5)--(.72,-.5)--(1.33,-1)--(.66,-1);
\draw (2,0) ellipse (.8 and 1.5); \draw node at (2,0) {\scriptsize $2$};  \draw node at (2,2) {\scriptsize $V_{2}$};
\draw (2.66,1)--(3.33,1)--(2.72,.5)--(3.28,.5)--(2.8,0)--(3.2,0)--(2.72,-.5)--(3.25,-.5)--(2.66,-1)--(3.33,-1);
\draw (3.33,1)--(2.66,1)--(3.28,.5)--(2.72,.5)--(3.2,0)--(2.8,0)--(3.28,-.5)--(2.72,-.5)--(3.33,-1)--(2.66,-1);
\draw (4,0) ellipse (.8 and 1.5); \draw node at (4,0) {\scriptsize $3$};  \draw node at (4,2) {\scriptsize $V_{3}$};
\draw (4.66,1)--(5.33,1)--(4.72,.5)--(5.28,.5)--(4.8,0)--(5.2,0)--(4.72,-.5)--(5.25,-.5)--(4.66,-1)--(5.33,-1);
\draw (5.33,1)--(4.66,1)--(5.28,.5)--(4.72,.5)--(5.2,0)--(4.8,0)--(5.28,-.5)--(4.72,-.5)--(5.33,-1)--(4.66,-1);
\draw node at (6,0) {\scriptsize $\dots$};
\draw (6.66,1)--(7.33,1)--(6.72,.5)--(7.28,.5)--(6.8,0)--(7.2,0)--(6.72,-.5)--(7.25,-.5)--(6.66,-1)--(7.33,-1);
\draw (7.33,1)--(6.66,1)--(7.28,.5)--(6.72,.5)--(7.2,0)--(6.8,0)--(7.28,-.5)--(6.72,-.5)--(7.33,-1)--(6.66,-1);
\draw (8,0) ellipse (.8 and 1.5); \draw node at (8,0) {\tiny $c-1$};  \draw node at (8,2) {\scriptsize $V_{c-1}$};
\draw (8.66,1)--(9.33,1)--(8.72,.5)--(9.28,.5)--(8.8,0)--(9.2,0)--(8.72,-.5)--(9.25,-.5)--(8.66,-1)--(9.33,-1);
\draw (9.33,1)--(8.66,1)--(9.28,.5)--(8.72,.5)--(9.2,0)--(8.8,0)--(9.28,-.5)--(8.72,-.5)--(9.33,-1)--(8.66,-1);
\draw (10,0) ellipse (.8 and 1.5); \draw node at (10,0) {\scriptsize $0$};  \draw node at (10,2) {\scriptsize $V_{c}$};
\draw (10.66,1)--(11.33,1)--(10.72,.5)--(11.28,.5)--(10.8,0)--(11.2,0)--(10.72,-.5)--(11.25,-.5)--(10.66,-1)--(11.33,-1);
\draw (11.33,1)--(10.66,1)--(11.28,.5)--(10.72,.5)--(11.2,0)--(10.8,0)--(11.28,-.5)--(10.72,-.5)--(11.33,-1)--(10.66,-1);
\draw (12,0) ellipse (.8 and 1.5); \draw node at (12,.75) {\scriptsize $1$};  \draw node at (12,2) {\scriptsize $V_{c+1}$};
\draw[fill] (12,-1) circle (0.1); \draw node at (12,-.5) {\scriptsize $v$};
\draw (12.66,1)--(13.33,1)--(12.72,.5)--(13.28,.5)--(12.8,0)--(13.2,0)--(12.72,-.5)--(13.25,-.5)--(12.66,-1)--(13.33,-1);
\draw (13.33,1)--(12.66,1)--(13.28,.5)--(12.72,.5)--(13.2,0)--(12.8,0)--(13.28,-.5)--(12.72,-.5)--(13.33,-1)--(12.66,-1);
\draw (14,0) ellipse (.8 and 1.5); \draw node at (14,0) {\scriptsize $2$};  \draw node at (14,2) {\scriptsize $V_{c+2}$};
\draw (14.66,1)--(15.33,1)--(14.72,.5)--(15.28,.5)--(14.8,0)--(15.2,0)--(14.72,-.5)--(15.25,-.5)--(14.66,-1)--(15.33,-1);
\draw (15.33,1)--(14.66,1)--(15.28,.5)--(14.72,.5)--(15.2,0)--(14.8,0)--(15.28,-.5)--(14.72,-.5)--(15.33,-1)--(14.66,-1);
\draw (16,0) ellipse (.8 and 1.5); \draw node at (16,0) {\scriptsize $3$};  \draw node at (16,2) {\scriptsize $V_{c+3}$};
\draw (16.66,1)--(17.33,1)--(16.72,.5)--(17.28,.5)--(16.8,0)--(17.2,0)--(16.72,-.5)--(17.25,-.5)--(16.66,-1)--(17.33,-1);
\draw (17.33,1)--(16.66,1)--(17.28,.5)--(16.72,.5)--(17.2,0)--(16.8,0)--(17.28,-.5)--(16.72,-.5)--(17.33,-1)--(16.66,-1);
\draw node at (18,0) {\scriptsize $\dots$};
\draw (18.66,1)--(19.33,1)--(18.72,.5)--(19.28,.5)--(18.8,0)--(19.2,0)--(18.72,-.5)--(19.25,-.5)--(18.66,-1)--(19.33,-1);
\draw (19.33,1)--(18.66,1)--(19.28,.5)--(18.72,.5)--(19.2,0)--(18.8,0)--(19.28,-.5)--(18.72,-.5)--(19.33,-1)--(18.66,-1);
\draw (20,0) ellipse (.8 and 1.5); \draw node at (20,0) {\scriptsize $0$};  \draw node at (20,2) {\scriptsize $V_{2c}$};
\draw (20.66,1)--(21.33,1)--(20.72,.5)--(21.28,.5)--(20.8,0)--(21.2,0)--(20.72,-.5)--(21.25,-.5)--(20.66,-1)--(21.33,-1);
\draw (21.33,1)--(20.66,1)--(21.28,.5)--(20.72,.5)--(21.2,0)--(20.8,0)--(21.28,-.5)--(20.72,-.5)--(21.33,-1)--(20.66,-1);
\draw (22,0) ellipse (.8 and 1.5); \draw node at (22,0) {\scriptsize $1$};  \draw node at (22,2) {\scriptsize $V_{2c+1}$};
\draw (22.66,1)--(23.33,1)--(22.72,.5)--(23.28,.5)--(22.8,0)--(23.2,0)--(22.72,-.5)--(23.25,-.5)--(22.66,-1)--(23.33,-1);
\draw (23.33,1)--(22.66,1)--(23.28,.5)--(22.72,.5)--(23.2,0)--(22.8,0)--(23.28,-.5)--(22.72,-.5)--(23.33,-1)--(22.66,-1);
\draw (24,0) ellipse (.8 and 1.5); \draw node at (24,0) {\scriptsize $2$};  \draw node at (24,2) {\scriptsize $V_{2c+2}$};
\draw (24.66,1)--(25.33,1)--(24.72,.5)--(25.28,.5)--(24.8,0)--(25.2,0)--(24.72,-.5)--(25.25,-.5)--(24.66,-1)--(25.33,-1);
\draw (25.33,1)--(24.66,1)--(25.28,.5)--(24.72,.5)--(25.2,0)--(24.8,0)--(25.28,-.5)--(24.72,-.5)--(25.33,-1)--(24.66,-1);
\draw node at (26,0) {\scriptsize $\dots$};
\draw (26.66,1)--(27.33,1)--(26.72,.5)--(27.28,.5)--(26.8,0)--(27.2,0)--(26.72,-.5)--(27.25,-.5)--(26.66,-1)--(27.33,-1);
\draw (27.33,1)--(26.66,1)--(27.28,.5)--(26.72,.5)--(27.2,0)--(26.8,0)--(27.28,-.5)--(26.72,-.5)--(27.33,-1)--(26.66,-1);
\draw (28.2,0) ellipse (.8 and 1.5); \draw node at (28.2,0) {$t_c$};  \draw node at (28.2,2) {\scriptsize $V_{t}$};

\draw [decorate,decoration={brace,amplitude=10, mirror},xshift=0,yshift=0] (9,-1.5) -- (20.8,-1.5) node [black,midway,yshift=-15] {\scriptsize $c-1$ moves};
\draw [decorate,decoration={brace,amplitude=10, mirror},xshift=0,yshift=0] (1,-2.5) -- (21,-2.5) node [black,midway,yshift=-15] {\scriptsize $c-2$ moves};
\draw [decorate,decoration={brace,amplitude=10, mirror},xshift=0,yshift=0] (-1,-3.5) -- (23,-3.5) node [black,midway,yshift=-15] {\scriptsize $1$ move};
\draw [decorate,decoration={brace,amplitude=10, mirror},xshift=0,yshift=0] (-1,-4.5) -- (29.1,-4.5) node [black,midway,yshift=-15] {\scriptsize $t-(2c+1)$ moves};

\draw[dashed] (-1,1.7)--(-1,-4.5); \draw[dashed] (1,1.7)--(1,-2.5); \draw[dashed] (9,1.7)--(9,-1.5); \draw[dashed] (21,1.7)--(21,-2.5); \draw[dashed] (23,1.7)--(23,-3.5); \draw[dashed] (29.1,1.7)--(29.1,-4.5);

\end{tikzpicture}
\caption{The second base case for Lemma \ref{rainbow-blowup}}
\label{base-case-2}
\end{figure}
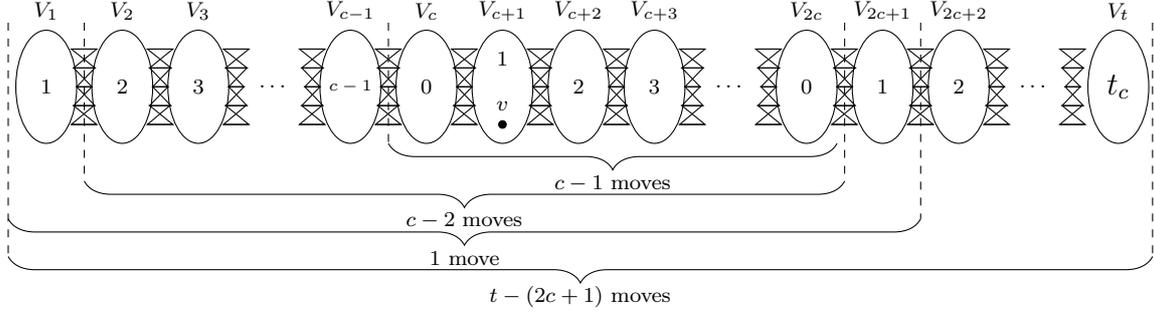

For the final base case, suppose that $t = 3c + 1$.  In this case we play $t - 4 = t - \left \lceil \frac{t}{c} \right \rceil$ moves, all at a vertex $v \in V_{c+2}$, as illustrated in Figure \ref{base-case-3}.  We begin by giving $v$ colours $3,\ldots,c-1,0,1$ in turn, which creates a monochromatic component containing all of $V_{c+1} \cup \cdots \cup V_{2c+1}$ except for vertices in $V_{c+2}$ having colour $2$.  We play a further $c-1$ moves in this component, giving it colours $0,c-1,\ldots,2$ in turn, which creates a monochromatic component containing all of $V_2 \cup \cdots \cup V_{2c+2}$.  Finally, we give this component colours $3,\ldots,c-1,0,1$ in turn, which floods the remainder of the graph.  Thus we can flood $G$ with a total of $c-1 + c - 1 + c - 1 = 3c - 3 = t - 4$ moves, as required.

\begin{figure}[h]
\centering
\begin{tikzpicture}[scale=.45]
\draw (0,0) ellipse (.8 and 1.5); \draw node at (0,0) {\scriptsize $1$};  \draw node at (0,2) {\scriptsize $V_{1}$};
\draw (.66,1)--(1.33,1)--(.72,.5)--(1.28,.5)--(.8,0)--(1.2,0)--(.72,-.5)--(1.25,-.5)--(.66,-1)--(1.33,-1);
\draw (1.33,1)--(.66,1)--(1.28,.5)--(.72,.5)--(1.2,0)--(.8,0)--(1.28,-.5)--(.72,-.5)--(1.33,-1)--(.66,-1);
\draw (2,0) ellipse (.8 and 1.5); \draw node at (2,0) {\scriptsize $2$};  \draw node at (2,2) {\scriptsize $V_{2}$};
\draw (2.66,1)--(3.33,1)--(2.72,.5)--(3.28,.5)--(2.8,0)--(3.2,0)--(2.72,-.5)--(3.25,-.5)--(2.66,-1)--(3.33,-1);
\draw (3.33,1)--(2.66,1)--(3.28,.5)--(2.72,.5)--(3.2,0)--(2.8,0)--(3.28,-.5)--(2.72,-.5)--(3.33,-1)--(2.66,-1);
\draw node at (4,0) {\scriptsize $\dots$};
\draw (4.66,1)--(5.33,1)--(4.72,.5)--(5.28,.5)--(4.8,0)--(5.2,0)--(4.72,-.5)--(5.25,-.5)--(4.66,-1)--(5.33,-1);
\draw (5.33,1)--(4.66,1)--(5.28,.5)--(4.72,.5)--(5.2,0)--(4.8,0)--(5.28,-.5)--(4.72,-.5)--(5.33,-1)--(4.66,-1);
\draw (6,0) ellipse (.8 and 1.5); \draw node at (6,0) {\tiny $c-1$};  \draw node at (6,2) {\scriptsize $V_{c-1}$};
\draw (6.66,1)--(7.33,1)--(6.72,.5)--(7.28,.5)--(6.8,0)--(7.2,0)--(6.72,-.5)--(7.25,-.5)--(6.66,-1)--(7.33,-1);
\draw (7.33,1)--(6.66,1)--(7.28,.5)--(6.72,.5)--(7.2,0)--(6.8,0)--(7.28,-.5)--(6.72,-.5)--(7.33,-1)--(6.66,-1);
\draw (8,0) ellipse (.8 and 1.5); \draw node at (8,0) {\scriptsize $0$};  \draw node at (8,2) {\scriptsize $V_{c}$};
\draw (8.66,1)--(9.33,1)--(8.72,.5)--(9.28,.5)--(8.8,0)--(9.2,0)--(8.72,-.5)--(9.25,-.5)--(8.66,-1)--(9.33,-1);
\draw (9.33,1)--(8.66,1)--(9.28,.5)--(8.72,.5)--(9.2,0)--(8.8,0)--(9.28,-.5)--(8.72,-.5)--(9.33,-1)--(8.66,-1);
\draw (10,0) ellipse (.8 and 1.5); \draw node at (10,0) {\scriptsize $1$};  \draw node at (10,2) {\scriptsize $V_{c+1}$};
\draw (10.66,1)--(11.33,1)--(10.72,.5)--(11.28,.5)--(10.8,0)--(11.2,0)--(10.72,-.5)--(11.25,-.5)--(10.66,-1)--(11.33,-1);
\draw (11.33,1)--(10.66,1)--(11.28,.5)--(10.72,.5)--(11.2,0)--(10.8,0)--(11.28,-.5)--(10.72,-.5)--(11.33,-1)--(10.66,-1);
\draw (12,0) ellipse (.8 and 1.5); \draw node at (12,.75) {\scriptsize $2$};  \draw node at (12,2) {\scriptsize $V_{c+2}$};
\draw[fill] (12,-1) circle (0.1); \draw node at (12,-.5) {\scriptsize $v$};
\draw (12.66,1)--(13.33,1)--(12.72,.5)--(13.28,.5)--(12.8,0)--(13.2,0)--(12.72,-.5)--(13.25,-.5)--(12.66,-1)--(13.33,-1);
\draw (13.33,1)--(12.66,1)--(13.28,.5)--(12.72,.5)--(13.2,0)--(12.8,0)--(13.28,-.5)--(12.72,-.5)--(13.33,-1)--(12.66,-1);
\draw node at (14,0) {\scriptsize $\dots$};
\draw (14.66,1)--(15.33,1)--(14.72,.5)--(15.28,.5)--(14.8,0)--(15.2,0)--(14.72,-.5)--(15.25,-.5)--(14.66,-1)--(15.33,-1);
\draw (15.33,1)--(14.66,1)--(15.28,.5)--(14.72,.5)--(15.2,0)--(14.8,0)--(15.28,-.5)--(14.72,-.5)--(15.33,-1)--(14.66,-1);
\draw (16,0) ellipse (.8 and 1.5); \draw node at (16,0) {\tiny $c-1$};  \draw node at (16,2) {\scriptsize $V_{2c-1}$};
\draw (16.66,1)--(17.33,1)--(16.72,.5)--(17.28,.5)--(16.8,0)--(17.2,0)--(16.72,-.5)--(17.25,-.5)--(16.66,-1)--(17.33,-1);
\draw (17.33,1)--(16.66,1)--(17.28,.5)--(16.72,.5)--(17.2,0)--(16.8,0)--(17.28,-.5)--(16.72,-.5)--(17.33,-1)--(16.66,-1);
\draw (18,0) ellipse (.8 and 1.5); \draw node at (18,0) {\scriptsize $0$};  \draw node at (18,2) {\scriptsize $V_{2c}$};
\draw (18.66,1)--(19.33,1)--(18.72,.5)--(19.28,.5)--(18.8,0)--(19.2,0)--(18.72,-.5)--(19.25,-.5)--(18.66,-1)--(19.33,-1);
\draw (19.33,1)--(18.66,1)--(19.28,.5)--(18.72,.5)--(19.2,0)--(18.8,0)--(19.28,-.5)--(18.72,-.5)--(19.33,-1)--(18.66,-1);
\draw (20,0) ellipse (.8 and 1.5); \draw node at (20,0) {\scriptsize $1$};  \draw node at (20,2) {\scriptsize $V_{2c+1}$};
\draw (20.66,1)--(21.33,1)--(20.72,.5)--(21.28,.5)--(20.8,0)--(21.2,0)--(20.72,-.5)--(21.25,-.5)--(20.66,-1)--(21.33,-1);
\draw (21.33,1)--(20.66,1)--(21.28,.5)--(20.72,.5)--(21.2,0)--(20.8,0)--(21.28,-.5)--(20.72,-.5)--(21.33,-1)--(20.66,-1);
\draw (22,0) ellipse (.8 and 1.5); \draw node at (22,0) {\scriptsize $2$};  \draw node at (22,2) {\scriptsize $V_{2c+2}$};
\draw (22.66,1)--(23.33,1)--(22.72,.5)--(23.28,.5)--(22.8,0)--(23.2,0)--(22.72,-.5)--(23.25,-.5)--(22.66,-1)--(23.33,-1);
\draw (23.33,1)--(22.66,1)--(23.28,.5)--(22.72,.5)--(23.2,0)--(22.8,0)--(23.28,-.5)--(22.72,-.5)--(23.33,-1)--(22.66,-1);
\draw node at (24,0) {\scriptsize $\dots$};
\draw (24.66,1)--(25.33,1)--(24.72,.5)--(25.28,.5)--(24.8,0)--(25.2,0)--(24.72,-.5)--(25.25,-.5)--(24.66,-1)--(25.33,-1);
\draw (25.33,1)--(24.66,1)--(25.28,.5)--(24.72,.5)--(25.2,0)--(24.8,0)--(25.28,-.5)--(24.72,-.5)--(25.33,-1)--(24.66,-1);
\draw (26,0) ellipse (.8 and 1.5); \draw node at (26,0) {\tiny $c-1$};  \draw node at (26,2) {\scriptsize $V_{3c-1}$};
\draw (26.66,1)--(27.33,1)--(26.72,.5)--(27.28,.5)--(26.8,0)--(27.2,0)--(26.72,-.5)--(27.25,-.5)--(26.66,-1)--(27.33,-1);
\draw (27.33,1)--(26.66,1)--(27.28,.5)--(26.72,.5)--(27.2,0)--(26.8,0)--(27.28,-.5)--(26.72,-.5)--(27.33,-1)--(26.66,-1);
\draw (28,0) ellipse (.8 and 1.5); \draw node at (28,0) {\scriptsize $0$};  \draw node at (28,2) {\scriptsize $V_{3c}$};
\draw (28.66,1)--(29.33,1)--(28.72,.5)--(29.28,.5)--(28.8,0)--(29.2,0)--(28.72,-.5)--(29.25,-.5)--(28.66,-1)--(29.33,-1);
\draw (29.33,1)--(28.66,1)--(29.28,.5)--(28.72,.5)--(29.2,0)--(28.8,0)--(29.28,-.5)--(28.72,-.5)--(29.33,-1)--(28.66,-1);
\draw (30,0) ellipse (.8 and 1.5); \draw node at (30,0) {\scriptsize $1$};  \draw node at (30,2) {\scriptsize $V_{3c+1}$};
\draw [decorate,decoration={brace,amplitude=10, mirror},xshift=0,yshift=0] (9,-1.5) -- (21,-1.5) node [black,midway,yshift=-15] {\scriptsize $c-1$ moves};
\draw [decorate,decoration={brace,amplitude=10, mirror},xshift=0,yshift=0] (1,-2.75) -- (23,-2.75) node [black,midway,yshift=-15] {\scriptsize $c-1$ moves};
\draw [decorate,decoration={brace,amplitude=10, mirror},xshift=0,yshift=0] (-0.9,-4) -- (31,-4) node [black,midway,yshift=-15] {\scriptsize $c-1$ moves};
\draw[dashed] (-1,1.7)--(-1,-4); \draw[dashed] (1,1.7)--(1,-2.75); \draw[dashed] (9,1.7)--(9,-1.5); \draw[dashed] (21,1.7)--(21,-1.5); \draw[dashed] (23,1.7)--(23,-2.75); \draw[dashed] (31,1.7)--(31,-4);
\end{tikzpicture}

\caption{The third base case for Lemma \ref{rainbow-blowup}}
\label{base-case-3}
\end{figure}
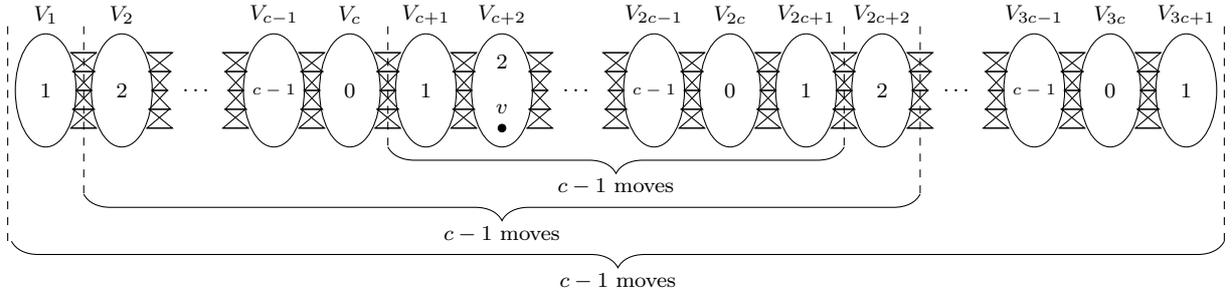

From now on, therefore, we may assume that $t \geq 3c + 2$, and that the result holds for any graph that is a blow-up of $P_s$ for $s < t$.  Using a similar strategy to that described in the second base case above, we can play $2c - 2$ moves which create a monochromatic component in $G$ containing all of $V_1 \cup \cdots \cup V_{2c + 1}$.  To achieve this, we play $2c-1$ moves at a vertex $v \in V_{c+1}$: we give this vertex colours $2,\ldots,c-1,0$, followed by $c-1,\ldots,2,1$.  Note that the resulting monochromatic component ends up with colour $1$, and that playing the sequence of moves described above in $G$ will not change the colour of any vertex outside $V_1 \cup \cdots \cup V_{2c+1}$.  Thus, after playing this sequence, the resulting coloured graph is equivalent to a graph $G'$ with colouring $\omega'$, where $G'$ is a blow-up of the path $P_{t-2c}$ and $\omega'$ is a $C$-rainbow colouring of $G'$.  Note that $t - 2c \geq c + 2$, so we can apply the inductive hypothesis to see that 
$$m(G',\omega') \leq t - 2c - \left \lceil \frac{t-2c}{c} \right \rceil = t - 2c + 2 - \left \lceil \frac{t}{c} \right \rceil.$$
This implies that 
$$m(G,\omega) \leq 2c - 2 + t - 2c + 2 - \left \lceil \frac{t}{c} \right \rceil = t - \left \lceil \frac{t}{c} \right \rceil,$$
as required.
\end{proof}

\subsection{General path colourings}
\label{path-col}

Before proving that the upper bound is also valid for path blow-ups with any path colouring (provided that the path is sufficiently long compared with the number of colours), we need some auxiliary results.  First of all, it is straightforward to verify the following characterisation of path colourings that are \emph{not} rainbow colourings.

\begin{prop}
Let $G$ be a blow-up of the path $P_t$, let $f: \{1,\ldots,t\} \rightarrow \{0,\ldots,c-1\}$ be any function and $C = \{d_0,\ldots,d_{c-1}\}$ a set of colours, and let $\omega$ be defined by setting $\omega(u) = d_{f(i)}$ for all $u \in V_i$ (for $1 \leq i \leq t$).  If $\omega$ is not a $C$-rainbow colouring of $G$, then there exists $1 \leq i < j \leq t$ such that $j - i < c$ and $f(i) = f(j)$.
\label{not-rainbow-col}
\end{prop}

We also need another result relating the number of moves required to flood a path and a collection of subpaths.

\begin{lma}
Let $P$ be a path on $t$ vertices with colouring $\omega$ from colour-set $C$, where $|C| = c$, and let $Q_1,\ldots,Q_r$ be a collection of disjoint subpaths of $P$.  Then
$$m(P,\omega) \leq t - \sum_{i=1}^r(|Q_i| - 1) - \left \lceil \frac{t - \sum_{i=1}^r(|Q_i| - 1)}{c} \right \rceil + \sum_{i=1}^r m(Q_i,\omega).$$
\label{path-section}
\end{lma}
\begin{proof}
For each $1 \leq i \leq r$, fix $d_i \in C$ such that $m(Q_i,\omega) = m(Q_i,\omega,d)$.  We now define a new colouring $\omega'$ of $P$ by setting
\begin{equation*}
\omega'(v) = \begin{cases}
						d_i			& \text{if $v \in Q_i$} \\
						\omega(v)	& \text{otherwise.} 
			 \end{cases}
\end{equation*}
Observe that $P$ with colouring $\omega'$ is equivalent to a path on at most $t - \sum_{i=1}^r (|Q_i| - 1)$ vertices, so by Corollary \ref{path-result} we have 
$$m(P,\omega') \leq t - \sum_{i=1}^r(|Q_i| - 1) - \left\lceil \frac{t - \sum_{i=1}^r(|Q_i| - 1)}{c} \right\rceil.$$
Let $\mathcal{A}$ be the set of maximal monochromatic components of $P$ with respect to $\omega'$, where each $A \in \mathcal{A}$ has colour $d_A$ under $\omega'$.  Then, by Lemma \ref{change-colouring}, we have
\begin{align*}
m(P,\omega) & \leq m(P,\omega') + \sum_{A \in \mathcal{A}} m(A,\omega,d_A) \\
			& \leq t - \sum_{i=1}^r(|Q_i| - 1) - \left\lceil{ \frac{t - \sum_{i=1}^r(|Q_i| - 1)}{c} }\right\rceil + \sum_{A \in \mathcal{A}} m(A,\omega,d_A).
\end{align*}
So it remains to check that $\sum_{A \in \mathcal{A}} m(A,\omega,d_A) \leq \sum_{i=1}^r m(Q_i,\omega)$.  Note that it is possible that more than one of the subpaths $Q_1,\ldots,Q_r$ belongs to the same maximal monochromatic component with respect to $\omega'$; suppose that $A_1,\ldots,A_s$ are the elements of $\mathcal{A}$ that contain at least one subpath $Q_i$, and observe therefore that $s \leq r$.  Moreover, it is clear that, for each $1 \leq i \leq s$, 
$$m(A_i,\omega,d_{A_i}) \leq \sum_{Q_j \subseteq A_i} m(Q_j,\omega,d_j) = \sum_{Q_j \subseteq A_i} m(Q_j,\omega).$$
Observe also that, for $A \in \mathcal{A}$ with $A \notin \{A_1,\ldots,A_s\}$, $A$ is also a monochromatic component of $P$ with respect to $\omega$, so we have $m(A,\omega,d_A) = 0$. Hence, as no subpath $Q_j$ belongs to more than one monochromatic component $A_i$, we have 
$$\sum_{A \in \mathcal{A}} m(A,\omega,d_A) \leq \sum_{i=1}^s \sum_{Q_j \subseteq A_i} m(Q_i,\omega) = \sum_{i=1}^r m(Q_i,\omega),$$
completing the proof.
\end{proof}

We now use these auxiliary results to extend our upper bound to cover all initial colourings that are path colourings.  The key idea of the proof is to define a quantity that captures in a sense how far away the initial colouring is from a rainbow colouring.  If the initial colouring is sufficiently different from a rainbow colouring, we can argue that we must be able to create a monochromatic end-to-end path significantly more quickly than in the rainbow case, meaning that we are then able to flood any remaining vertices greedily.  In the event that the colouring does not differ so much from a rainbow colouring, we demonstrate how we may perform a sequence of flooding moves that is not too long and which results in a rainbow-coloured path blow-up, allowing us to apply the previous result.

\begin{lma}
Let $G$ be a blow-up of the path $P_t$, let $c \geq 3$ and let $f: \{1,\ldots,t\} \rightarrow \{1,\ldots,c\}$ be any function, and let $\omega$ be defined by setting $\omega(u) = f(i)$ for all $u \in V_i$ (for $1 \leq i \leq t$).  Then, if $t \geq 2c^2(c-1)^3$,
$$m(G,\omega) \leq t - \left \lceil \frac{t}{c} \right \rceil.$$
\label{path-colouring}
\end{lma}
\begin{proof}
Without loss of generality, we may assume that, for every $1 \leq x < t$, $f(x) \neq f(x+1)$, as otherwise we could contract all vertices of $V_x \cup V_{x+1}$ to a single vertex, obtaining an equivalent coloured graph which is a blow-up of a path on $t - 1$ vertices.  Now observe that, for any path colouring $\omega$ of $G$, the graph $G$ can be decomposed into subgraphs $G_1,\ldots,G_r$, where each $G_i$ is a blow-up of the path $P_{t_i}$ and $\sum_{i=1}^r t_i = t$, in such a way that $\omega$ is a $C$-rainbow colouring of $G_i$ for each $1 \leq i \leq r$; it is clear that such a decomposition must exist since setting $G_i = G[V_i]$ for $1 \leq i \leq t$ will do.

A more meaningful decomposition with the required properties can be constructed greedily, as illustrated in Figure \ref{greedy-decomp}.  We set $s_1 = 1$ and choose $t_1$ to be the largest integer such that $\omega$ is a rainbow colouring of $G[V_1 \cup \cdots \cup V_{t_1}]$; given $t_1,\ldots,t_i$, we set $s_{i+1} = 1 + \sum_{j=1}^{i} t_j$ and choose $t_{i+1}$ to be the largest integer such that $\omega$ is a rainbow colouring of $G[V_{s_{i+1}} \cup \cdots \cup V_{s_{i+1} + t_{i+1} - 1}]$.  We call the decomposition constructed in this way the \emph{greedy rainbow decomposition of $(G,\omega)$}, and denote by $\grd(G,\omega)$ the collection of subgraphs in this decomposition.  Then $\grd(G,\omega) = \{G_1,\ldots,G_r\}$ (for some $r \geq 1$), where $G_i = G[V_{s_i} \cup V_{s_i+1} \cup \cdots \cup V_{s_{i+1} - 1}]$, and the greedy construction guarantees that, for each $1 \leq i \leq r = |\grd(G,\omega)|$, there exists $x_i$ with $\max\{s_i,s_{i+1} - c + 1\} \leq x_i \leq s_{i+1} - 1$ such that $f(x_i) = f(s_{i+1})$.

\begin{figure}[h]
\centering
\begin{tikzpicture}[scale=.42]
\draw (0,0) ellipse (.8 and 1.5); \draw node at (0,0) {\scriptsize $1$}; \draw node at (0,2) {\scriptsize $V_{1}$};
\draw (.66,1)--(1.33,1)--(.72,.5)--(1.28,.5)--(.8,0)--(1.2,0)--(.72,-.5)--(1.25,-.5)--(.66,-1)--(1.33,-1);
\draw (1.33,1)--(.66,1)--(1.28,.5)--(.72,.5)--(1.2,0)--(.8,0)--(1.28,-.5)--(.72,-.5)--(1.33,-1)--(.66,-1);
\draw (2,0) ellipse (.8 and 1.5); \draw node at (2,0) {\scriptsize $2$}; \draw node at (2,2) {\scriptsize $V_{2}$};
\draw (2.66,1)--(3.33,1)--(2.72,.5)--(3.28,.5)--(2.8,0)--(3.2,0)--(2.72,-.5)--(3.25,-.5)--(2.66,-1)--(3.33,-1);
\draw (3.33,1)--(2.66,1)--(3.28,.5)--(2.72,.5)--(3.2,0)--(2.8,0)--(3.28,-.5)--(2.72,-.5)--(3.33,-1)--(2.66,-1);
\draw (4,0) ellipse (.8 and 1.5); \draw node at (4,0) {\scriptsize $3$}; \draw node at (4,2) {\scriptsize $V_{3}$};
\draw (4.66,1)--(5.33,1)--(4.72,.5)--(5.28,.5)--(4.8,0)--(5.2,0)--(4.72,-.5)--(5.25,-.5)--(4.66,-1)--(5.33,-1);
\draw (5.33,1)--(4.66,1)--(5.28,.5)--(4.72,.5)--(5.2,0)--(4.8,0)--(5.28,-.5)--(4.72,-.5)--(5.33,-1)--(4.66,-1);
\draw (6,0) ellipse (.8 and 1.5); \draw node at (6,0) {\scriptsize $2$}; \draw node at (6,2) {\scriptsize $V_{4}$};
\draw (6.66,1)--(7.33,1)--(6.72,.5)--(7.28,.5)--(6.8,0)--(7.2,0)--(6.72,-.5)--(7.25,-.5)--(6.66,-1)--(7.33,-1);
\draw (7.33,1)--(6.66,1)--(7.28,.5)--(6.72,.5)--(7.2,0)--(6.8,0)--(7.28,-.5)--(6.72,-.5)--(7.33,-1)--(6.66,-1);
\draw (8,0) ellipse (.8 and 1.5); \draw node at (8,0) {\scriptsize $3$}; \draw node at (8,2) {\scriptsize $V_{5}$};
\draw (8.66,1)--(9.33,1)--(8.72,.5)--(9.28,.5)--(8.8,0)--(9.2,0)--(8.72,-.5)--(9.25,-.5)--(8.66,-1)--(9.33,-1);
\draw (9.33,1)--(8.66,1)--(9.28,.5)--(8.72,.5)--(9.2,0)--(8.8,0)--(9.28,-.5)--(8.72,-.5)--(9.33,-1)--(8.66,-1);
\draw (10,0) ellipse (.8 and 1.5); \draw node at (10,0) {\scriptsize $4$}; \draw node at (10,2) {\scriptsize $V_{6}$};
\draw (10.66,1)--(11.33,1)--(10.72,.5)--(11.28,.5)--(10.8,0)--(11.2,0)--(10.72,-.5)--(11.25,-.5)--(10.66,-1)--(11.33,-1);
\draw (11.33,1)--(10.66,1)--(11.28,.5)--(10.72,.5)--(11.2,0)--(10.8,0)--(11.28,-.5)--(10.72,-.5)--(11.33,-1)--(10.66,-1);
\draw (12,0) ellipse (.8 and 1.5); \draw node at (12,0) {\scriptsize $5$}; \draw node at (12,2) {\scriptsize $V_{7}$};
\draw (12.66,1)--(13.33,1)--(12.72,.5)--(13.28,.5)--(12.8,0)--(13.2,0)--(12.72,-.5)--(13.25,-.5)--(12.66,-1)--(13.33,-1);
\draw (13.33,1)--(12.66,1)--(13.28,.5)--(12.72,.5)--(13.2,0)--(12.8,0)--(13.28,-.5)--(12.72,-.5)--(13.33,-1)--(12.66,-1);
\draw (14,0) ellipse (.8 and 1.5); \draw node at (14,0) {\scriptsize $1$}; \draw node at (14,2) {\scriptsize $V_{8}$};
\draw (14.66,1)--(15.33,1)--(14.72,.5)--(15.28,.5)--(14.8,0)--(15.2,0)--(14.72,-.5)--(15.25,-.5)--(14.66,-1)--(15.33,-1);
\draw (15.33,1)--(14.66,1)--(15.28,.5)--(14.72,.5)--(15.2,0)--(14.8,0)--(15.28,-.5)--(14.72,-.5)--(15.33,-1)--(14.66,-1);
\draw (16,0) ellipse (.8 and 1.5); \draw node at (16,0) {\scriptsize $2$}; \draw node at (16,2) {\scriptsize $V_{9}$};
\draw (16.66,1)--(17.33,1)--(16.72,.5)--(17.28,.5)--(16.8,0)--(17.2,0)--(16.72,-.5)--(17.25,-.5)--(16.66,-1)--(17.33,-1);
\draw (17.33,1)--(16.66,1)--(17.28,.5)--(16.72,.5)--(17.2,0)--(16.8,0)--(17.28,-.5)--(16.72,-.5)--(17.33,-1)--(16.66,-1);
\draw (18,0) ellipse (.8 and 1.5); \draw node at (18,0) {\scriptsize $3$}; \draw node at (18,2) {\scriptsize $V_{10}$};
\draw (18.66,1)--(19.33,1)--(18.72,.5)--(19.28,.5)--(18.8,0)--(19.2,0)--(18.72,-.5)--(19.25,-.5)--(18.66,-1)--(19.33,-1);
\draw (19.33,1)--(18.66,1)--(19.28,.5)--(18.72,.5)--(19.2,0)--(18.8,0)--(19.28,-.5)--(18.72,-.5)--(19.33,-1)--(18.66,-1);
\draw (20,0) ellipse (.8 and 1.5); \draw node at (20,0) {\scriptsize $4$}; \draw node at (20,2) {\scriptsize $V_{11}$};
\draw (20.66,1)--(21.33,1)--(20.72,.5)--(21.28,.5)--(20.8,0)--(21.2,0)--(20.72,-.5)--(21.25,-.5)--(20.66,-1)--(21.33,-1);
\draw (21.33,1)--(20.66,1)--(21.28,.5)--(20.72,.5)--(21.2,0)--(20.8,0)--(21.28,-.5)--(20.72,-.5)--(21.33,-1)--(20.66,-1);
\draw (22,0) ellipse (.8 and 1.5); \draw node at (22,0) {\scriptsize $2$}; \draw node at (22,2) {\scriptsize $V_{12}$};
\draw (22.66,1)--(23.33,1)--(22.72,.5)--(23.28,.5)--(22.8,0)--(23.2,0)--(22.72,-.5)--(23.25,-.5)--(22.66,-1)--(23.33,-1);
\draw (23.33,1)--(22.66,1)--(23.28,.5)--(22.72,.5)--(23.2,0)--(22.8,0)--(23.28,-.5)--(22.72,-.5)--(23.33,-1)--(22.66,-1);
\draw (24,0) ellipse (.8 and 1.5); \draw node at (24,0) {\scriptsize $4$}; \draw node at (24,2) {\scriptsize $V_{13}$};
\draw (24.66,1)--(25.33,1)--(24.72,.5)--(25.28,.5)--(24.8,0)--(25.2,0)--(24.72,-.5)--(25.25,-.5)--(24.66,-1)--(25.33,-1);
\draw (25.33,1)--(24.66,1)--(25.28,.5)--(24.72,.5)--(25.2,0)--(24.8,0)--(25.28,-.5)--(24.72,-.5)--(25.33,-1)--(24.66,-1);
\draw (26,0) ellipse (.8 and 1.5); \draw node at (26,0) {\scriptsize $5$}; \draw node at (26,2) {\scriptsize $V_{14}$};
\draw (26.66,1)--(27.33,1)--(26.72,.5)--(27.28,.5)--(26.8,0)--(27.2,0)--(26.72,-.5)--(27.25,-.5)--(26.66,-1)--(27.33,-1);
\draw (27.33,1)--(26.66,1)--(27.28,.5)--(26.72,.5)--(27.2,0)--(26.8,0)--(27.28,-.5)--(26.72,-.5)--(27.33,-1)--(26.66,-1);
\draw (28,0) ellipse (.8 and 1.5); \draw node at (28,0) {\scriptsize $1$}; \draw node at (28,2) {\scriptsize $V_{15}$};
\draw (28.66,1)--(29.33,1)--(28.72,.5)--(29.28,.5)--(28.8,0)--(29.2,0)--(28.72,-.5)--(29.25,-.5)--(28.66,-1)--(29.33,-1);
\draw (29.33,1)--(28.66,1)--(29.28,.5)--(28.72,.5)--(29.2,0)--(28.8,0)--(29.28,-.5)--(28.72,-.5)--(29.33,-1)--(28.66,-1);
\draw (30,0) ellipse (.8 and 1.5); \draw node at (30,0) {\scriptsize $3$}; \draw node at (30,2) {\scriptsize $V_{16}$};
\draw (30.66,1)--(31.33,1)--(30.72,.5)--(31.28,.5)--(30.8,0)--(31.2,0)--(30.72,-.5)--(31.25,-.5)--(30.66,-1)--(31.33,-1);
\draw (31.33,1)--(30.66,1)--(31.28,.5)--(30.72,.5)--(31.2,0)--(30.8,0)--(31.28,-.5)--(30.72,-.5)--(31.33,-1)--(30.66,-1);
\draw (32,0) ellipse (.8 and 1.5); \draw node at (32,0) {\scriptsize $2$}; \draw node at (32,2) {\scriptsize $V_{17}$};
\draw [decorate,decoration={brace,amplitude=10, mirror},xshift=0,yshift=0] (-0.9,-1.5) -- (4.9,-1.5) node [black,midway,yshift=-15] {\scriptsize $G_1$};
\draw [decorate,decoration={brace,amplitude=10, mirror},xshift=0,yshift=0] (5.1,-1.5) -- (20.9,-1.5) node [black,midway,yshift=-15] {\scriptsize $G_2$};
\draw [decorate,decoration={brace,amplitude=10, mirror},xshift=0,yshift=0] (21.1,-1.5) -- (32.9,-1.5) node [black,midway,yshift=-15] {\scriptsize $G_3$};
\end{tikzpicture}
\caption{An example of a greedy rainbow decomposition}
\label{greedy-decomp}
\end{figure}
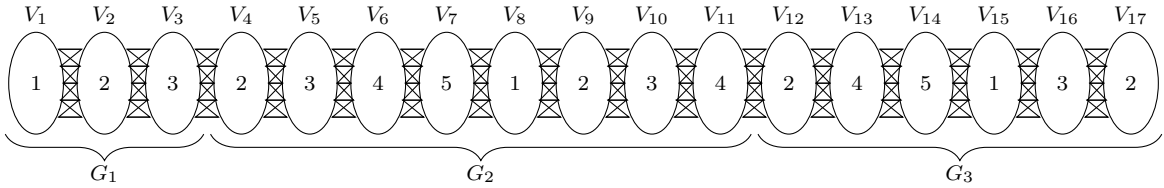

Suppose first that $|\grd(G,\omega)| > 2c(c-1)$.  We will argue that in this case we can flood $G$ in at most $m(G,\omega)$ moves by first creating a monochromatic end-to-end path and then cycling through any remaining colours.

Fix $Q$ to be any path that contains precisely one vertex from each vertex class $V_1,\ldots,V_t$.  Now, for each $1 \leq i \leq r-1$, set $Q_i$ to be the segment of $Q$ induced by \mbox{$Q \cap (V_{x_i} \cup \cdots \cup V_{s_{i+1}})$}.  Observe that for each $Q_i$, by definition of $x_i$, we have $|Q_i| \leq c$, and $m(Q_i,\omega,f(x_i)) \leq |Q_i| - 2$, by Proposition \ref{max-colour-class}.  We are not quite able to apply Lemma \ref{path-section}, as for any $i$ it is possible that $Q_i$ and $Q_{i+1}$ intersect in one vertex; however, it is clear that $Q_i \cap Q_{i+2} = \emptyset$ for any $i$, so it is certainly the case that $\{Q_{2i}: 1 \leq i \leq \left \lfloor \frac{|\grd(G,\omega)|}{2} \right \rfloor\}$ is a collection of disjoint subpaths of $Q$.  Setting $r = |\grd(G,\omega)|$, Lemma \ref{path-section} now tells us that
\begin{align*}
m(Q,\omega) & \leq t - \sum_{i=1}^{\left \lfloor \frac{r}{2} \right \rfloor} (|Q_{2i}| - 1) - \left \lceil \frac{t - \sum_{i=1}^{\left \lfloor \frac{r}{2} \right \rfloor}(|Q_{2i}| - 1)}{c} \right \rceil + \sum_{i=1}^{\left \lfloor \frac{r}{2} \right \rfloor} m(Q_{2i},\omega) \\
			   & \leq t - \left \lceil \frac{t}{c} \right \rceil - (c-1),
\end{align*}
since $|Q_i| \leq c$, $m(Q_{2i},\omega) \leq |Q_i| - 2$, and $r=|\grd(G,\omega)| \geq 2c(c-1)$.  Proposition \ref{dominating-path} then gives $m(G,\omega) \leq t - \left \lceil \frac{t}{c} \right \rceil$, as required.

So we may assume from now on that $|\grd(G,\omega)| \leq 2c(c-1)$.  In this case it suffices to prove the following claim, by our initial assumption that $t \geq 2c^2(c-1)^3$.

\begin{claim*}
If $t \geq c(c-1)^2|\grd(G,\omega)|$ then
$$m(G,\omega) \leq t - \left \lceil \frac{t}{c} \right \rceil.$$
\label{grd-induction}
\end{claim*}

The base case, for $|\grd(G,\omega)| = 1$, follows from Lemma \ref{rainbow-blowup} (since in this case $\omega$ must be a rainbow colouring of $G$), so we may assume from now on that $|\grd(G,\omega)| \geq 2$ and that the claim holds for any such graph $\widetilde{G}$ with colouring $\widetilde{\omega}$ such that $|\grd(\widetilde{G},\widetilde{\omega})| < |\grd(G,\omega)|$.

By our assumption that $t \geq c(c-1)^2|\grd(G,\omega)|$, there must be some $1 \leq i \leq |\grd(G,\omega)|$ such that $t_i \geq c(c-1)^2 > 2c$.  Note that we may assume without loss of generality that $i \neq |\grd(G,\omega)|$: if the only such subgraph in the decomposition is $G_{|\grd(G,\omega)|}$ then we can reverse the ordering of the vertex classes so that we instead have $t_1 > 2c$ after re-ordering (the subgraphs of the decomposition may not be the same as before, but our longest section can only increase in length in this new setting).  We will now consider the subgraph $H_0 = G[V(G_i) \cup V(G_{i+1})]$, where $\omega_0$ is the restriction of $\omega$ to $V(H_0)$.  

We will describe how to play a sequence of moves in $H_0$ which results in a rainbow colouring of this subgraph; if this does not decrease the length of the underlying path too much (when monochromatic components are contracted) we then invoke the inductive hypothesis, and otherwise we can complete the proof directly.  Specifically, we describe how to obtain a sequence of graphs $(H_j)_{j=1}^s$ (for some $s \geq 1$), with corresponding colourings $(\omega_j)_{j=1}^s$, with three key properties.  Note that $H_0$ itself is not included in this sequence; it does not have the second property listed below.

We denote by $U_1^{(0)},\ldots,U_{\ell_0}^{(0)}$ the vertex classes of $H_0$, where $\ell_0 = t_i + t_{i+1}$, and let $f_0: \{1,\ldots,\ell_0\} \rightarrow C$ be the function such that, for every $1 \leq z \leq \ell_0$, $\omega_0(u) = f_0(z)$ for each $u \in U_z^{(0)}$.  For each $j$, the coloured graph $(H_j,\omega_j)$ then has the following properties:
\begin{enumerate}
\item $H_j$ is a blow-up of a path $P_{\ell_j}$, where $\ell_j \geq \ell_0 - (j+1)(c-1) - 1$, and $\omega_j$ is a proper path colouring of $H_j$,
\item There exists $x_j \in \{1,\ldots,\ell_j\}$ such that, if the vertex classes of $H_j$ are $U_1^{(j)},\ldots,U_{\ell_j}^{(j)}$, then $|U_{x_j}^{(j)}| = 1$ and $\omega_j$ is a $C$-rainbow colouring of both $H_j[U_1^{(j)} \cup \cdots \cup U_{x_j}^{(j)}]$ and $H_j[U_{x_j}^{(j)} \cup \cdots \cup U_{\ell_j}^{(j)}]$ (as illustrated in Figure \ref{H_j-structure}), and
\item If the colouring $\omega_j'$ of $V(H_0)$ is defined by
\begin{equation*}
\omega_j'(v) = \begin{cases}
					\omega_0(v) 		& \text{if $v \in U_z^{(0)}$ and either $z < x_j$ or $z > \ell_0 - (\ell_j - x_j)$} \\
					\omega_j(x_j)		& \text{otherwise,}
				\end{cases}
\end{equation*}
then $H_0$ with colouring $\omega_j'$ is equivalent to $H_j$ with colouring $\omega_j$; moreover, if $\mathcal{A}_j$ is the set of maximal monochromatic components of $H_0$ with respect to $\omega_j'$, where each $A \in \mathcal{A}_j$ has colour $d_A^j$ under $\omega_j'$, then $\sum_{A \in \mathcal{A}_j} m(A,\omega_0,d_A^j) \leq \ell_0 - \ell_j - j - 1$.
\end{enumerate}

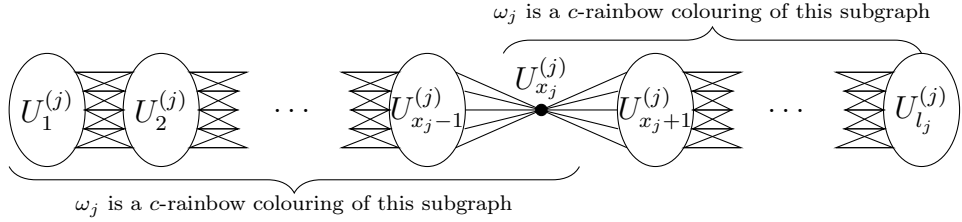
\begin{figure}[h]
\centering
\begin{tikzpicture}[scale=.5]
\draw (0,0) ellipse (1 and 1.5); \draw node at (0,0) {$U_1^{(j)}$};
\draw (.75,1)--(2.25,1)--(.98,.5)--(2.02,.5)--(1,0)--(2,0)--(.98,-.5)--(2.02,-.5)--(.75,-1)--(2.25,-1);
\draw (2.25,1)--(.75,1)--(2.02,.5)--(.98,.5)--(2,0)--(1,0)--(2.02,-.5)--(.98,-.5)--(2.25,-1)--(.75,-1);
\draw (3,0) ellipse (1 and 1.5); \draw node at (3,0) {$U_2^{(j)}$};
\draw (3.75,1)--(5.25,1)--(3.98,.5)--(5.02,.5)--(4,0)--(5,0)--(3.98,-.5)--(5.02,-.5)--(3.75,-1)--(5.25,-1);
\draw (5.25,1)--(3.75,1)--(5.02,.5)--(3.98,.5)--(5,0)--(4,0)--(5.02,-.5)--(3.98,-.5)--(5.25,-1)--(3.75,-1);
\draw node at (6.5,0) {$\dots$};
\draw (7.75,1)--(9.25,1)--(7.98,.5)--(9.02,.5)--(8,0)--(9,0)--(7.98,-.5)--(9.02,-.5)--(7.75,-1)--(9.25,-1);
\draw (9.25,1)--(7.75,1)--(9.02,.5)--(7.98,.5)--(9,0)--(8,0)--(9.02,-.5)--(7.98,-.5)--(9.25,-1)--(7.75,-1);
\draw (10,0) ellipse (1 and 1.5); \draw node at (10,0) {$U_{x_j-1}^{(j)}$};
\draw (10.75,1)--(13,0);\draw (10.98,.5)--(13,0);\draw (11,0)--(13,0);\draw (10.98,-.5)--(13,0);\draw (10.75,-1)--(13,0);
\draw[fill] (13,0) circle(0.15) node[above] {$U_{x_j}^{(j)}$};
\draw (15.25,1)--(13,0);\draw (15.02,.5)--(13,0);\draw (15,0)--(13,0);\draw (15.02,-.5)--(13,0);\draw (15.25,-1)--(13,0);
\draw (16,0) ellipse (1 and 1.5); \draw node at (16,0) {$U_{x_j+1}^{(j)}$};
\draw (16.75,1)--(18.25,1)--(16.98,.5)--(18.02,.5)--(17,0)--(18,0)--(16.98,-.5)--(18.02,-.5)--(16.75,-1)--(18.25,-1);
\draw (18.25,1)--(16.75,1)--(18.02,.5)--(16.98,.5)--(18,0)--(17,0)--(18.02,-.5)--(16.98,-.5)--(18.25,-1)--(16.75,-1);
\draw node at (19.5,0) {$\dots$};
\draw (20.75,1)--(22.25,1)--(20.98,.5)--(22.02,.5)--(21,0)--(22,0)--(20.98,-.5)--(22.02,-.5)--(20.75,-1)--(22.25,-1);
\draw (22.25,1)--(20.75,1)--(22.02,.5)--(20.98,.5)--(22,0)--(21,0)--(22.02,-.5)--(20.98,-.5)--(22.25,-1)--(20.75,-1);
\draw (23,0) ellipse (1 and 1.5); \draw node at (23,0) {$U_{l_j}^{(j)}$};

\draw [decorate,decoration={brace,amplitude=10, mirror},xshift=0,yshift=0] (-1,-1.5) -- (14,-1.5) node [black,midway,yshift=-15] {\scriptsize $\omega_j$ is a $c$-rainbow colouring of this subgraph};

\draw [decorate,decoration={brace,amplitude=10},xshift=0,yshift=0] (12, 1.5) -- (23,1.5) node [black,midway,yshift=15] {\scriptsize $\omega_j$ is a $c$-rainbow colouring of this subgraph};

\end{tikzpicture}
\caption{The structure of $H_j$}
\label{H_j-structure}
\end{figure}

Informally, we repeatedly perform moves to create a single monochromatic component at the boundary of the two rainbow-coloured segments (whose colour is consistent with both rainbow colourings), taking care to make sure that we do not play too many moves or shorten the path too much at any stage.  The key idea is to exploit the fact that some colour must occur more frequently, as we traverse the subgraph from left to right, than would happen under a rainbow colouring.

We describe in detail how to obtain the first pair $(H_1,\omega_1)$; the method for constructing further pairs is very similar but somewhat simpler.  Throughout, the only assumption we require, in addition to the fact that $(H_j,\omega_j)$ has the three stated properties, is that $\omega_j$ is not a rainbow colouring of $H_j$.

By construction of $\grd(G,\omega)$, we know that there is some $y \in \{t_i - c + 2, \ldots, t_i\}$ such that $f_0(y) = f_0(t_i+1)$: if not, then we would have chosen $G_i$ to include at least one more vertex class.  In fact, by our assumption that no two consecutive vertex classes receive the same colour under $\omega$, we know that $y \in \{t_i-c+2,\ldots,t_i-1\}$.  Since $\omega_0$ is a $C$-rainbow colouring of $G_i$, and $t_i > 2c$, we also know that $f_0(y-c) = f_0(y) = f_0(t_i+1)$.  Now set $F_0 = H_0[U_{y-c}^{(0)} \cup \cdots \cup U_{t_i+1}^{(0)}]$.  We now describe a sequence of moves to flood $F_0$ with colour $f(y)$ in at most $t_i + c - y - 1$ moves, all played at some vertex $v_0 \in U_y^{(0)}$; this sequence of moves is illustrated in Figure \ref{flood_F_0}.  We begin by giving $v_0$ colours $f_0(y-1),\ldots,f_0(y-c+1) = f(y+1)$ in turn; this will create a monochromatic component containing $v_0$ and all of $U_{y-c+1}^{(0)} \cup \cdots \cup U_{y-1}^{(0)} \cup U_{y+1}^{(0)}$, and so that the only vertices of $U_{y-c+1}^{(0)} \cup \cdots \cup U_{y+1}^{(0)}$ not linked to this component have colour $f_0(y)$.  We then give this component colours $f_0(y+2),\ldots,f_0(t_i+1) = f_0(y) = f_0(y-c)$ in turn, which will clearly flood all remaining vertices in $F_0$.  The total number of moves played is therefore $c-1 + t_i - y = t_i + c - y - 1$, implying that $m(F_0,\omega_0,f_0(y)) \leq t_i + c - y - 1$.

\begin{figure}[h]
\centering
\begin{tikzpicture}[scale=.4]
\draw (0,0) ellipse (2 and 1.5); \draw node at (0,0) {$f_0(y-c)$};
\draw (1.49,1)--(3.51,1)--(1.76,.5)--(3.24,.5)--(2,0)--(3,0)--(1.76,-.5)--(3.24,-.5)--(1.49,-1)--(3.51,-1);
\draw (3.51,1)--(1.49,1)--(3.24,.5)--(1.76,.5)--(3,0)--(2,0)--(3.24,-.5)--(1.76,-.5)--(3.51,-1)--(1.49,-1);
\draw (5,0) ellipse (2 and 1.5); \draw node at (5,0) {\scriptsize $f_0(y-c+1)$};
\draw (6.49,1)--(8.51,1)--(6.76,.5)--(8.24,.5)--(7,0)--(8,0)--(6.76,-.5)--(8.24,-.5)--(6.49,-1)--(8.51,-1);
\draw (8.51,1)--(6.49,1)--(8.24,.5)--(6.76,.5)--(8,0)--(7,0)--(8.24,-.5)--(6.76,-.5)--(8.51,-1)--(6.49,-1);
\draw node at (9,0) {$\dots$};
\draw (9.49,1)--(11.51,1)--(9.76,.5)--(11.24,.5)--(10,0)--(11,0)--(9.76,-.5)--(11.24,-.5)--(9.49,-1)--(11.51,-1);
\draw (11.51,1)--(9.49,1)--(11.24,.5)--(9.76,.5)--(11,0)--(10,0)--(11.24,-.5)--(9.76,-.5)--(11.51,-1)--(9.49,-1);
\draw (13,0) ellipse (2 and 1.5); \draw node at (13,0) {$f_0(y)$};
\draw (14.49,1)--(16.51,1)--(14.76,.5)--(16.24,.5)--(15,0)--(16,0)--(14.76,-.5)--(16.24,-.5)--(14.49,-1)--(16.51,-1);
\draw (16.51,1)--(14.49,1)--(16.24,.5)--(14.76,.5)--(16,0)--(15,0)--(16.24,-.5)--(14.76,-.5)--(16.51,-1)--(14.49,-1);
\draw (18,0) ellipse (2 and 1.5); \draw node at (18,0) {$f_0(y+1)$};
\draw (19.49,1)--(21.51,1)--(19.76,.5)--(21.24,.5)--(20,0)--(21,0)--(19.76,-.5)--(21.24,-.5)--(19.49,-1)--(21.51,-1);
\draw (21.51,1)--(19.49,1)--(21.24,.5)--(19.76,.5)--(21,0)--(20,0)--(21.24,-.5)--(19.76,-.5)--(21.51,-1)--(19.49,-1);
\draw node at (22,0) {$\dots$};
\draw (22.49,1)--(24.51,1)--(22.76,.5)--(24.24,.5)--(23,0)--(24,0)--(22.76,-.5)--(24.24,-.5)--(22.49,-1)--(24.51,-1);
\draw (24.51,1)--(22.49,1)--(24.24,.5)--(22.76,.5)--(24,0)--(23,0)--(24.24,-.5)--(22.76,-.5)--(24.51,-1)--(22.49,-1);
\draw (26,0) ellipse (2 and 1.5); \draw node at (26,0) {$f_0(t_i)$};
\draw (27.49,1)--(29.51,1)--(27.76,.5)--(29.24,.5)--(28,0)--(29,0)--(27.76,-.5)--(29.24,-.5)--(27.49,-1)--(29.51,-1);
\draw (29.51,1)--(27.49,1)--(29.24,.5)--(27.76,.5)--(29,0)--(28,0)--(29.24,-.5)--(27.76,-.5)--(29.51,-1)--(27.49,-1);
\draw (31,0) ellipse (2 and 1.5); \draw node at (31,0) {$f_0(t_i+1)$};

\draw [decorate,decoration={brace,amplitude=10, mirror},xshift=0,yshift=0] (2.5,-1.5) -- (20.5,-1.5) node [black,midway,yshift=-15] {\scriptsize $c-1$ moves};
\draw [decorate,decoration={brace,amplitude=10, mirror},xshift=0,yshift=0] (2.5,-2.75) -- (28.5,-2.75) node [black,midway,yshift=-15] {\scriptsize $t_i-y-1$ moves};
\draw [decorate,decoration={brace,amplitude=10, mirror},xshift=0,yshift=0] (-2.2,-4) -- (33.2,-4) node [black,midway,yshift=-15] {\scriptsize $1$ move};

\draw[dashed] (-2.2,1.7)--(-2.2,-4); \draw[dashed] (2.5,1.7)--(2.5,-2.75); \draw[dashed] (20.5,1.7)--(20.5,-1.5); \draw[dashed] (28.5,1.7)--(28.5,-2.75); \draw[dashed] (33.2,1.7)--(33.2,-4);
\end{tikzpicture}
\caption{Flooding the subgraph $F_0$}
\label{flood_F_0}
\end{figure}
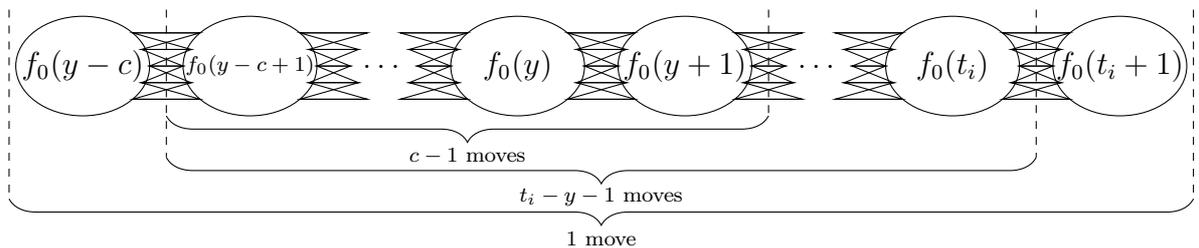

Now define $H_1$ to be the graph obtained from $H_0$ by contracting all vertices of $F_0$ to a single vertex $w_1$, and $\omega_1$ to be the colouring of $H_1$ that agrees with $\omega_0$ on all vertices of $H_1$ except $w_1$, and gives $w_1$ colour $f_0(y)$.  We claim that $H_1$ with colouring $\omega_1$ has the three properties listed above.  It is clear that $H_1$ is a blow-up of a path on 
$$\ell_1 = \ell_0 - (t_i+1 - y + c) \geq \ell_0 - 2c + 1 = \ell_0 - (1+1)(c-1) - 1$$
vertices and that $\omega_1$ is a proper path colouring of $H_1$, as required to satisfy the first condition.  For the second condition, set $x_1 = y-c$, and note that $U_{x_1}^{(1)} = \{w_1\}$.  Further define $f_1: \{1,\ldots,\ell_1\} \rightarrow C$ to be the function so that $\omega_1(u) = f_1(z)$ for every $u \in U_z^{(1)}$, for $1 \leq z \leq \ell_1$.  Observe that $f_1(z) = f_0(z)$ for $1 \leq z \leq y-c$, so it follows from the fact that $\omega_0$ is a $C$-rainbow colouring of $H_0[U_1^{(0)} \cup \cdots \cup U_{t_i}^{(0)}] \supset H_0[U_1^{(0)} \cup \cdots \cup U_{y-c}^{(0)}]$ that $\omega_1$ is a $C$-rainbow colouring of $H_1[U_1^{(1)} \cup \cdots \cup U_{x_1}^{(1)}]$.  Moreover, for $y-c \leq z \leq \ell_1$, we see that $f_1(z) = f_0(z + \ell_0 - \ell_1)$, so the fact that $\omega_0$ is a $C$-rainbow colouring of $H_0[U_{t_i+1}^{(0)} \cup \cdots \cup U_{\ell_0}^{(0)}]$ implies that $\omega_1$ is a $C$-rainbow colouring of 
$$H_1[U_{t_i+1-(\ell_0 - \ell_1)}^{(0)} \cup \cdots \cup U_{\ell_0 - (\ell_0 - \ell_1)}^{(0)}] = H_1[U_{x_1}^{(1)} \cup \cdots \cup U_{\ell_1}^{(1)}].$$
Thus the second condition holds.  Finally, for the third condition, it is clear from the definition of $\omega_1$ that $H_1$ with colouring $\omega_1$ is equivalent to the graph $H_0$ with colouring $\omega_1'$ (with $\omega_1'$ defined with respect to $\omega_1$ as in the statement of the third condition); note also that the only maximal monochromatic component of $H_0$ with respect to $\omega_1'$ that is not also a maximal monochromatic component with respect to $\omega_0$ is $F_0$ (and it is straightforward to verify that $F_0$ is indeed a maximal monochromatic component of $H_0$ with respect to $\omega_1'$).  Thus, if $\mathcal{A}_1$ denotes the set of maximal monochromatic components of $H_0$ with respect to $\omega_1'$ and each $A \in \mathcal{A}_1$ has colour $d_A$ under $\omega_1'$, we see that
\begin{align*}
\sum_{A \in \mathcal{A}_1} m(A,\omega_0,d_A) & = m(F_0,\omega_0,\omega_1(x_1)) \\
											   & = m(F_0,\omega_0,f_0(y)) \\
										  	   & \leq t_i + c - y - 1 \\
											   & = (t_i + 1 - y + c) - 2 \\
											   & = \ell_0 - \ell_1 - 2,
\end{align*} 
as required to satisfy the third condition.  This completes the definition of $H_1$ and $\omega_1$.

Observe that we can continue applying the same general procedure (omitting the part of the process that makes $F_0$ monochromatic) to obtain a new pair $(H_{j+1},\omega_{j+1})$ with the same three properties so long as $(H_j,\omega_j)$ has the three stated properties and $w_j$ is not a rainbow colouring of $H_j$.  Suppose that we construct a sequence of pairs $(H_j,\omega_j)$ for $1 \leq j \leq s$ in this way, where $s$ is as large as possible.  By maximality of $s$, we may assume that $\omega_s$ is a rainbow colouring of $H_s$, as otherwise we could continue.  (Note that if our colouring is not a rainbow colouring, this imposes a minimum condition on the length of the path, so we do not need to consider separately the possibility of our path becoming too short to apply the procedure.)  We define $\omega'$ to be the colouring of $G$ which agrees with $\omega$ on all vertices that do not belong to $H_0$, and with $\omega_s'$ on all vertices of $H_0$.  Note that the maximal monochromatic components of $G$ with respect to $\omega'$ that are \emph{not} also maximal monochromatic components with respect to $\omega$ are precisely the maximal monochromatic components of $H_0$ with respect to $\omega' = \omega_s'$ (and recall also that $\omega_0$ is the restriction of $\omega$ to $H_0$).  Thus we can apply Lemma \ref{change-colouring} to see that
\begin{align}
m(G,\omega) & \leq m(G,\omega') + \sum_{A \in \mathcal{A}_s} m(A,\omega_0,d_A^s) \nonumber \\
			  & \leq m(G,\omega') + \ell_0 - \ell_s - s - 1 \label{G,omega'}
\end{align}
by the third condition on $(H_s,\omega_s)$.  There are now two cases to consider, depending on the value of $s$.

First, suppose that $s \geq c(c-1)$.  In this case we argue that we can continue by creating a monochromatic end-to-end path and then cycling through any remaining colours.  Let $Q$ be a path in $G$ which contains precisely one vertex from each class.  Note that there will be a segment of $\ell_0 - \ell_s + 1$ consecutive vertices on $Q$ which have the same colour under $\omega'$ so, under this colouring, $Q$ is equivalent to a path of length $|Q| - \ell_0 + \ell_s$; Proposition \ref{dominating-path}, together with Corollary \ref{path-result}, therefore implies that 
$$m(G,\omega') \leq t - \ell_0 + \ell_s - \left \lceil \frac{t - \ell_0 + \ell_s}{c} \right \rceil + c - 1.$$
Substituting this in \eqref{G,omega'} and using both the fact that $\ell_s \geq \ell_0 - (s+1)(c-1) - 1$ (the first condition on $(H_s,\omega_s)$) and the assumption that $s \geq c(c-1)$, this gives
$$m(G,\omega) \leq t - \left \lceil \frac{t}{c} \right \rceil,$$
as required.

Now suppose instead that $s < c(c-1)$; in this case we invoke the inductive hypothesis.  Since $\omega_s$ is a $C$-rainbow colouring of $H_0$, it is clear that $|\grd(G,\omega')| < |\grd(G,\omega)|$.  Moreover, $G$ with colouring $\omega'$ is equivalent to a graph $\widetilde{G}$ with colouring $\widetilde{\omega}$, where $\widetilde{G}$ is a blow-up of a path on $r$ vertices with
\begin{align*}
r & = t - (\ell_0 - \ell_s) \\
  & > c(c-1)^2|\grd(G,\omega')|
\end{align*}
(making use of our assumption on the value of $t$ and the fact that $|\grd(G,\omega')| \leq |\grd(G,\omega)| - 1$).  Thus we can apply the inductive hypothesis to see that the claim holds for $G$ with colouring $\omega'$, implying that
$$m(G,\omega') \leq t - (\ell_0 - \ell_s) - \left \lceil \frac{t - (\ell_0 - \ell_s)}{c} \right \rceil.$$
Substituting into \eqref{G,omega'}, then gives the required result, completing the proof of the claim, and hence proving the result.
\end{proof}

\subsection{Arbitrary colourings}
\label{all-col}

In this section we show that our upper bound can be extended to all initial colourings.  The structure of this proof is in some ways similar to the previous result: we define a notion of the distance of a colouring from a path colouring, and then consider two cases.  If the colouring differs sufficiently from a path colouring, we can quickly create an end-to-end path and flood any remaining vertices greedily, whereas if our initial colouring is sufficiently close to a path colouring we demonstrate how to play a sequence of moves that results in a path-coloured graph, allowing us to apply the previous result.

\begin{lma}
Let $G$ be a blow-up of the path $P_t$, where $t \geq 2c^{10}$.  Then
$$M_c(G,\omega) \leq t - \left \lceil \frac{t}{c} \right \rceil.$$
\label{path-blowup-ub}
\end{lma}
\begin{proof}
let $\omega$ be any colouring of $G$ from colour-set $C = \{1,\ldots,c\}$.  We begin by setting 
$$\theta(G,\omega) = |\{i: 1 \leq i \leq t \text{ and } \omega \text{ is not constant on } V_i\}|.$$
Suppose first that $\theta(G,\omega) \geq c(c-1)$, and set
$$n_j = |\{i: 1 \leq i \leq t \text{ and } \exists u \in V_i \text{ with } \omega(u) = j \}.$$
Since there are $\theta(G,\omega)$ vertex classes that each contain vertices of at least two distinct colours, we see that 
$$\sum_{j = 1}^c n_j \geq t + \theta(G,\omega) \geq t + c(c-1).$$
Thus there exists some $j \in \{1,\ldots,c\}$ such that $n_j \geq \left \lceil \frac{t}{c} \right \rceil + (c-1)$.  Observe therefore that there exists a path $Q$ containing precisely one vertex from each vertex class $V_1, \ldots, V_t$ and so that at least $\left \lceil \frac{t}{c} \right \rceil + (c-1)$ vertices on $Q$ have colour $j$ under $\omega$.  It then follows from Proposition \ref{max-colour-class} that $m(Q,\omega) \leq t - \left \lceil \frac{t}{c} \right \rceil - (c-1)$, so Proposition \ref{dominating-path} gives
$$m(G,\omega) \leq t - \left \lceil \frac{t}{c} \right \rceil - (c-1) + (c-1) = t - \left \lceil \frac{t}{c} \right \rceil,$$
as required.

Thus from now on we will assume that $\theta(G,\omega) < c(c-1)$.  In this case it clearly suffices to prove the following claim, since we are assuming that $t \geq 2c^{10} > 2c^8(\theta(G,\omega) + 1)$.

\begin{claim*}
Suppose that $t > 2c^8(\theta(G,\omega) + 1)$.  Then
$$m(G,\omega) \leq t - \left \lceil \frac{t}{c} \right \rceil.$$
\label{theta-ind}
\end{claim*}

We prove the claim by induction on $\theta(G,\omega)$.  In the base case, for $\theta(G,\omega) = 0$, we know that $\omega$ must in fact be a path-colouring of $G$ and so the result follows immediately from Lemma \ref{path-colouring}.  Thus we may assume that $\theta(G,\omega) \geq 1$ and that the result holds for any graph $G'$ with colouring $\omega'$ such that $\theta(G',\omega') < \theta(G,\omega)$.

Since $t \geq 2c^8(\theta(G,\omega) + 1)$, there exists some vertex class $V_i$ such that $\omega$ is not constant on $V_i$, but for $1 \leq j \leq 2c^8$ we either have $\omega$ constant on every $V_{i+j}$, or else $\omega$ is constant on every $V_{i-j}$; reversing the order of the vertex classes if necessary, we may assume without loss of generality that we have $\omega$ constant on every $V_{i+j}$ for $1 \leq j \leq 2c^8$.  

The first step in our strategy to flood $G$ is to perform a series of moves in the \mbox{$(2c^5 + 1)(c^2(c-1) + 1)$} classes to the right of $V_i$, resulting in a colouring of these vertices that makes the subgraph they induce (after contracting monochromatic components) equivalent to a path.  Note that $\omega$ defines a path colouring on the subgraph induced by these classes.

We split this subgraph into $c^2(c-1)+1$ consecutive blocks, $B_1,\ldots,B_{c^2(c-1)+1}$, where each block consists of the subgraph induced by $2c^5+1$ consecutive vertex classes.  As the restriction of $\omega$ to any block $B_{\ell}$ is a path colouring, it follows from Lemma \ref{path-colouring} that $B_{\ell}$ can be made monochromatic in some colour $d_{\ell}$ with at most 
$$2c^5+1 - \left\lceil \frac{2c^5+1}{c} \right\rceil = 2c^5 - 2c^4$$
moves.  Let $\overline{\omega}$ be the colouring of $V(G)$ that assigns $d_{\ell}$ to every vertex of $B_{\ell}$ (for $1 \leq \ell \leq c^2(c-1) + 1$), and agrees with $\omega$ elsewhere.  Then Lemma \ref{change-colouring} tells us that 
\begin{equation}
m(G,\omega) \leq m(G,\overline{\omega}) + (c^2(c-1)+1)(2c^5 - 2c^4). \label{amended-target}
\end{equation}

We now consider the graph $G'$, obtained by contracting monochromatic components of $G$ with respect to $\overline{\omega}$, and the corresponding colouring $\omega'$.  Note that $G'$ is a blow-up of a path on $t' \leq t - 2c^5(c^2(c-1)+1)$ vertices.  In the remainder of the proof we will argue that in fact $m(G',\omega') \leq t' - \left \lceil \frac{t'}{c} \right \rceil$; it is straightforward to check that substituting this bound on $m(G,\overline{\omega})$ in \eqref{amended-target} gives the required result.

We now prove this bound on $m(G',\omega')$.  Recall from the construction of $G'$ that, if the vertex classes of $G'$ are $U_1,\ldots,U_{t'}$, we have $|U_{i + \ell}| = 1$ for $1 \leq \ell \leq c^2(c-1)+1$.  We may also assume without loss of generality that the colours assigned to vertices of $U_i$ by $\omega'$ are $\{1,\ldots,r\}$ for some $r \geq 2$.

If every colour in $\{1,\ldots,r\}$ is assigned to one of the first $c$ vertex classes to the right of $U_i$ by $\omega'$, then we can create a monochromatic component containing all of $U_i$ and the $c$ vertex classes to its right using at most $c-1$ moves: we give the unique vertex in $U_{i+1}$ the colours of $U_{i+2},U_{i+3},\ldots,U_{i+c}$ in turn.  Otherwise, there must be some $d \in \{1,\ldots,r\}$ that is not assigned to any of the first $c$ vertex-classes to the right of $U_i$ by $\omega'$.  Then, by Proposition \ref{max-colour-class}, we can perform an ``efficient flooding sequence'' in which we flood this subpath on $c$ vertices with at most $c-2$ moves (as some colour must be repeated), thus reducing the length of an end-to-end path by $c-1$.  We continue in this way until either we have performed an efficient flooding operation $c(c-1)$ times, or else we are able, by the method described above, to flood $U_i$ and the $c$ vertex classes immediately to the right with $c-1$ moves.

In the latter case, we have played $a(c-2)+(c-1)$ moves, for some $a <c(c-1)$, to create a coloured graph equivalent to a graph $\widetilde{G}$ with colouring $\widetilde{\omega}$, where $\widetilde{G}$ is a blow-up of a path on $t'' = t' - (a(c-1)+c)$ vertices and $\theta(\widetilde{G},\widetilde{\omega}) < \theta(G,\omega)$.\footnote{In fact the moves we have played might have flooded a larger component than is described here; but by Lemma \ref{change-colouring} this can only help us.}  It is straightforward to verify that $t'' \geq 2c^8(\theta(\widetilde{G},\widetilde{\omega})+1)$ and so it follows from the inductive hypothesis that 
$$m(\widetilde{G},\widetilde{\omega}) \leq t'' - \left \lceil \frac{t''}{c} \right \rceil = t' - (a(c-1)+c) - \left \lceil \frac{t' - (a(c-1)+c)}{c} \right \rceil;$$
the required bound on $m(G',\omega')$ then follows easily from the fact that $m(G',\omega') \leq a(c-2) + (c-1) + m(\widetilde{G},\widetilde{\omega})$.

It remains to consider the case that we terminate after performing $c(c-1)$ efficient flooding sequences.  In this case, we have (after contracting monochromatic components) reduced the length of an end-to-end path by $c(c-1)^2$, with at most $c(c-1)(c-2)$ moves.  Thus, by Proposition \ref{dominating-path} we can flood the resulting graph with a further 
$$t'-c(c-1)^2 - \left\lceil \frac{t' - c(c-1)^2}{c}\right\rceil + c - 1 = t' - \left\lceil \frac{t'}{c} \right\rceil - c(c-1)(c-2)$$
moves, from which the bound on $m(G',\omega')$ follows immediately, completing the proof.
\end{proof}

%%%%%%%%%% SECTION 6: CONCLUSION AND OPEN PROBLEMS %%%%%%%%%%
\section{Conclusions and Open Problems}

We have given several upper and lower bounds on the maximum number of moves, taken over all possible colourings, that may be required to flood a given graph $G$ in both the fixed and free variants of the game, and we have demonstrated that these bounds are tight for suitable families of trees.

Motivated by the intuition that adding a large number of edges to the graph should reduce the number of moves required in the worst case, we also demonstrate that the number of moves required in the worst case (in the free version) to flood a blow-up of a sufficiently long path is the same as the number required to flood a path of the same length.  If we add all possible edges to a tree of radius $r$ that do not reduce the length of a shortest path between any vertex $u$ and the vertex $v$ of minimum eccentricity, the resulting graph contains a blow-up of a path of length $r+1$ as a subgraph, so our result on path blow-ups demonstrates that adding edges in this way gives a dramatic reduction in the worst-case number of moves required to flood the graph.  It would be interesting to investigate precisely how many edges must be added to a tree of radius $r$ to decrease the worst-case number of move required.

Our results provide a partial answer to a question raised by Meeks and Scott in \cite{2xn}, that of determining the maximum number of moves, taken over all possible colourings, that may be required to flood a given graph $G$. This general question remains open, and given the emphasis on $k \times n$ grids in the existing algorithmic analysis of flood-filling games, a natural first direction for further research would be that of determining the exact value of $M_c(G)$ and $M_c^{(v)}(G)$ in the case that $G$ is a $k \times n$ grid. Using Proposition~\ref{max-colour-class} and Lemma~\ref{rainbow-target}, we only have $n - \left \lceil \frac{n}{c} \right \rceil$ and $n - \left \lceil \frac{n}{c} \right \rceil + (c-1) \left \lceil \frac{k-1}{2} \right \rceil$ as lower and upper bounds in the free version, respectively; the bounds for the fixed version are even worse.

The parameterised complexity of determining whether a given coloured graph can be flooded with a specified number of moves has been studied with a wide range of different parameterisations (as in, for example, \cite{fellows-flood15}), but this investigation of extremal properties gives rise to a new natural parameterised problem: given a graph $G$, for which we know the value of $M_c(G)$, and a colouring $\omega$ of $G$, what is the (parameterised) complexity of determining whether $G$ with initial colouring $\omega$ can be flooded in at most $M_c(G) - k$ moves, where $k$ is the parameter?  It can easily be deduced from the hardness proof for $2 \times n$ boards in \cite{2xn} that determining the minimum number of moves required to flood a blow-up of a path is NP-hard, so it is already meaningful to consider this parameterised problem for graphs drawn from the classes considered here.

%\nocite{*}
\bibliographystyle{abbrvnat}
% use the following instead if you encounter problems 
%\bibliographystyle{alpha}
\bibliography{../FIrefs}

\end{document}